\documentclass[a4paper, 10pt]{amsart}

\usepackage[utf8]{inputenc}
\usepackage[T1]{fontenc}
\usepackage[english]{babel}

\usepackage{amsmath, amssymb}
\usepackage{algorithm}
\usepackage{algpseudocode}
\usepackage{indentfirst}

\usepackage{graphicx}
\usepackage{subcaption}

\theoremstyle{plain}
\newtheorem{theorem}{Theorem}[section]

\newtheorem{proposition}[theorem]{Proposition}
\newtheorem{lemma}[theorem]{Lemma}
\theoremstyle{definition}
\newtheorem{definition}[theorem]{Definition}

\theoremstyle{remark}
\newtheorem{remark}[theorem]{Remark}

\algdef{SE}[DOWHILE]{Do}{doWhile}{\algorithmicdo}[1]{\algorithmicwhile\ #1}

\newcommand{\R}{\mathbb{R}}
\newcommand{\E}{\mathcal{E}}

\newcommand{\C}{\mathbb{C}}
\newcommand{\e}{\mathrm{E}}

\newcommand{\G}{\mathcal{G}}
\newcommand{\T}{\mathcal{T}}
\newcommand{\di}{\mathrm{d}}

\newcommand{\A}{\mathcal{A}}
\newcommand{\HH}{\mathcal{H}}
\newcommand{\Om}{\Omega}

\newcommand{\F}{\mathcal{F}}

\newcommand{\J}{\mathcal{J}}
\newcommand{\p}{\mathrm{P}}
\newcommand{\D}{\mathcal{D}}
\newcommand{\TOL}{\mathit{Tol}}

\newcommand{\N}{\mathbb{N}}

\newcommand{\coloneq }{:=}

\newcommand{\abs}[1]{\lvert #1 \rvert}
\newcommand{\bigabs}[1]{\bigl\lvert #1 \bigr\rvert}

\newcommand{\norm}[1]{\lVert #1 \rVert}
\newcommand{\bignorm}[1]{\bigl\lVert #1 \bigr\rVert}

\newcommand{\dd}{\mathrm{d}}

\newcommand{\tforall}{\quad \text{for all }}

\DeclareMathOperator*{\argmin}{arg\,min}

\numberwithin{equation}{section}

\begin{document}
		
\title[Consistent Approximation of Phase-Field Models]{Consistent Finite-Dimensional Approximation\\of Phase-Field Models of Fracture}

\author[S. Almi]{Stefano Almi}
\address{Fakult\"at für Mathematik, TUM\\Boltzmannstr. 3\\85748 Garching bei M\"unchen, Germany} \email{stefano.almi@ma.tum.de}

\author[S. Belz]{Sandro Belz}
\address{Fakult\"at für Mathematik, TUM\\Boltzmannstr. 3\\85748 Garching bei M\"unchen, Germany} \email{sandro.belz@ma.tum.de}


	
\begin{abstract}
	In this paper we focus on the finite-dimensional approximation of quasi-static evolutions of critical points of the phase-field model of brittle fracture. In a space discretized setting, we first discuss an alternating minimization scheme which, together with the usual time-discretization procedure, allows us to construct such finite-dimensional evolutions. Then, passing to the limit as the space discretization becomes finer and finer, we prove that any limit of a sequence of finite-dimensional evolutions is itself a quasi-static evolution of the phase-field model of fracture. In particular, our proof shows for the first time the consistency of numerical schemes related to the study of fracture mechanics and image processing.
\end{abstract}

\maketitle

\section{Introduction}\label{Introduction}

In this paper we are interested in the study of convergence of numerical schemes for quasi-static evolution of brittle fractures in elastic bodies. We focus on the phase-field (or damage) approximation of fracture studied by Bourdin, Francfort, and Marigo in~\cite{Bourdin2000,Bourdin2007a,Bourdin2008}, and first introduced by Ambrosio and Tortorelli in~\cite{Ambrosio1990a,Ambrosio1992} in the framework of image processing. 

In a planar setting, given an open bounded subset~$\Om$ of~$\R^{2}$ with Lipschitz boundary~$\partial\Om$, we deal with an energy functional of the form
\begin{equation}\label{AmbTort}
	\J_\varepsilon (u,v) \coloneq\frac{1}{2}\int_{\Omega} (v^2 + \eta_\varepsilon) \abs{\nabla u}^2 \, \dd x + \kappa \int_{\Omega} \varepsilon \abs{\nabla v}^2 \, \dd x+ \kappa\int_{\Om} \frac{(1-v)^2}{4\varepsilon} \, \dd x\,,
\end{equation}
where $\varepsilon$ and $\eta_{\varepsilon}$ are two small positive parameters, $u\in H^{1}(\Om)$ stands for the displacement field, $v\in H^{1}(\Om;[0,1])$ denotes the damage variable, and the positive constant~$\kappa$ may be interpreted as the toughness of the material, which we assume to be equal to one for the following discussion. From a physical point of view, the variable~$v$ in~\eqref{AmbTort} takes into account how damaged the elastic body is, so that, for $x\in\Om$, $v(x)=0$ means that the damage is complete (fracture) at~$x$, while $v(x)=1$ means that the material is perfectly intact at~$x$.

In~\cite{Ambrosio1990a,Ambrosio1992} it has been shown that choosing $0 < \eta_\varepsilon \ll \varepsilon$ and letting $\varepsilon\to0$, the functional~$\J_{\varepsilon}$ $\Gamma$-converges to
\begin{equation} \label{FrancfortFunc}
	\G(u)\coloneq\frac{1}{2}\int_{\Om} \abs{\nabla u}^2 \,\di x + \HH^{1} (S_{u}) \,,
\end{equation}
defined on~$GSBV(\Om)$, the space of generalized special function of bounded variation (for the theory of such spaces see, for instance,~\cite{Ambrosio2000}). In~\eqref{FrancfortFunc},~$\HH^{1}$ denotes the $1$-dimensional Hausdorff measure and~$S_{u}$ stands for the approximate discontinuity set of~$u$. In the mathematical model of fracture (see, e.g.,~\cite{Bourdin2008}), the functional~\eqref{FrancfortFunc} represents the energy of an elastic body~$\Om$ subject to an antiplanar displacement~$u$ and with a crack~$S_{u}$.

In view of such a convergence result, the phase-field functional~\eqref{AmbTort} has been widely and successfully used in numerical simulations of crack growth processes (see, for instance~\cite{Artina2015,Bourdin2000,Bourdin2007a,Bourdin2008,Burke2010,Burke2011}).

In the framework of numerical approaches to fracture mechanics, our interest is in the proof of consistency (or analysis of convergence) of some numerical schemes used in the study of the crack growth process. In particular, our goal is to prove the existence of quasi-static evolutions for the phase-field model~\eqref{AmbTort} as limits of evolutions obtained in a space-discretized setting. It is indeed clear that any numerical simulation based on~\eqref{AmbTort} gives as an outcome only \emph{finite-dimensional} approximations of an evolution of the phase-field variable (see, e.g.,~\cite{Bourdin2000,Bourdin2007a,Bourdin2008}). This is due to the fact that, in order to implement some kind of algorithm, a discretization of the functional space~$H^{1}(\Om)$ is needed. Having this in mind, our contribution is, roughly speaking, the following: we show that we can construct a quasi-static evolution of critical points of the phase-field model of fracture~\eqref{AmbTort} as a limit of \emph{finite-dimensional quasi-static evolutions} obtained in a discretized~$H^{1}$-framework. Clearly, the limit process is performed as the function space discretization becomes finer and finer. 

Up to our knowledge, this paper provides the first proof of consistency of numerical methods for such evolutions, going beyond the empirical consistency checks in~\cite{Artina2015b}. Moreover, we needed to fuse classical methods of PDE discretization, such as FEM (Finite Element Method) and their typical quasi-interpolating estimates, together with some variational techniques, going far beyond the usual linear setting where, e.g., FEM are employed.

In contrast to~\cite{Giacomini2005}, or even to the more general framework of~\cite{Mielke2008}, where the study of the limit of evolutions along global minimizers is performed, essentially, with $\Gamma$\nobreakdash-convergence techniques, and similarly to the works of Braides and coauthors in~\cite{Braides2011,Braides2008}, here we develop results of consistency for evolutions along critical points, which are both more realistic and technical novel.  

We anticipate here that all the results we are going to discuss are still valid in the vectorial case, i.e., when considering the functional
\begin{displaymath}
\mathcal{I}_{\varepsilon}(u,v)\coloneq\frac{1}{2}\int_{\Om}(v^{2}+\eta_{\varepsilon})\C\e u{\,\cdot\,}\e u\,\di x+\int_{\Omega}\varepsilon\abs{\nabla v}^2\,\dd x + \int_{\Om}\frac{(1-v)^2}{4\varepsilon} \, \dd x\,,
\end{displaymath}
where $u\in H^{1}(\Om;\R^{2})$, $v\in H^{1}(\Om)$, and~$\C$ is the usual elasticity tensor. For the sake of simplicity, we decided to present here in details only the scalar setting~\eqref{AmbTort}.

In order to be more precise in the discussion of our result, let us briefly present the quasi-static evolution problem we want to tackle in this work. For notational convenience, let us fix the parameters~$\varepsilon=\tfrac{1}{2}$ and $\eta_{\varepsilon}=\eta>0$ and let us drop the subscript~$\varepsilon$ in~\eqref{AmbTort}, so that we consider the functional
\begin{equation}\label{AmbTort2}
    \J(u,v) \coloneq\frac{1}{2} \int_{\Omega} (v^2 + \eta) \abs{\nabla u}^2 \, \dd x + \frac{1}{2}\int_{\Omega}\big( \abs{\nabla v}^2 +(1-v)^2\big) \, \dd x
\end{equation}
for $u,v\in H^{1}(\Om)$.
 Given $T>0$, we assume that the evolution of the elastic body~$\Om$ is driven by the energy functional~\eqref{AmbTort2} and by a time-dependent Dirichlet boundary datum $w\in W^{1,2}([0,T];H^{1}(\Om))$. In this context, a quasi-static evolution is described by the pair of functions $(u,v)\colon[0,T]\to H^{1}(\Om)\times H^{1}(\Om)$ standing for displacement and damage, respectively, and satisfying the following conditions (we refer to Definition~\ref{quasi-staticcontinuous} for a precise statement):
\begin{enumerate}
	\item \label{irreversibility} \emph{Irreversibility}: $0\leq v(t)\leq v(s)\leq 1$ a.e. in~$\Om$ for every $0\leq s\leq t\leq T$;
	
	\item \label{stabilCond} \emph{Stability}: for every $t\in[0,T]$, the pair $(u(t),v(t))$ is a ``critical point'' of the energy functional~$\J$ in the class of pairs $(u,v)\in H^{1}(\Om)\times H^{1}(\Om)$ such that $u=w(t)$ on~$\partial\Om$ and $v\leq v(t)$ a.e. in~$\Om$;
	
	\item \label{energyInequ}\emph{Energy-dissipation inequality}: for every $t\in[0,T]$
	\begin{equation*}
		\J(u(t),v(t))\leq\J(u(0),v(0))+\int_{0}^{t}\int_{\Om}(v^{2}(s)+\eta)\nabla{u(s)}\cdot\nabla{\dot{w}(s)}\,\di x\,\di s\,,
	\end{equation*}
	where $\dot{w}$ denotes the time derivative of $w$.
\end{enumerate}
We mention that this notion of evolution is often referred to as \emph{local energetic evolution}. See, e.g.,~\cite{Mielke2005} for further discussions on the topic.

The irreversibility property~\eqref{irreversibility} means that the damage process is unidirectional, in the sense that once the elastic body~$\Om$ is damaged, i.e.,~$v<1$ in a subset of~$\Om$, it can not be repaired, not even partially. We notice that this is the natural counterpart of the irreversibility of brittle fracture, which states that once a crack is created, it can not be closed anymore during the evolution process.

The stability condition~\eqref{stabilCond}, discussed in details in Section~\ref{Discrete}, can be mathematically rephrased, roughly speaking, as
\begin{equation}\label{Intro1}
	\partial_{(u,v)}\J(u(t),v(t))\leq0\,,
\end{equation}
where~$\partial_{(u,v)}$ denotes the partial derivative with respect to the pair of variables~$(u,v)$, and the inequality is due to the irreversibility constraint discussed above. As we will see in Section~\ref{Discrete}, inequality~\eqref{Intro1} can be splitted in
\begin{equation}\label{Intro2}
	\partial_{u}\J(u(t),v(t))=0\qquad\text{and}\qquad\partial_{v}\J(u(t),v(t))\leq0\,,
\end{equation}
or, which is equivalent because of the separate convexity of~$\J$ with respect to~$u$ and~$v$,
\begin{alignat}{2}
	\J(u(t),v(t)) &\leq\J(u,v(t))\qquad &&\text{for every $u\in H^{1}(\Om)$ with $u=w(t)$ on $\partial\Om$},\label{Intro3}\\
	\J(u(t),v(t)) &\leq \J(u(t),v)\qquad &&\text{for every $v\in H^{1}(\Om)$ with $v\leq v(t)$ a.e. in $\Om$}.\label{Intro4}
\end{alignat}
We notice that conditions~\eqref{Intro3}\nobreakdash--\eqref{Intro4} are not equivalent to the global stability property
\begin{equation*}
	\J(u(t),v(t))\leq\J(u,v)
\end{equation*}
for every pair $(u,v)\in H^{1}(\Om) \times H^1 (\Omega; [0,1])$ such that $u=w(t)$ on~$\partial\Om$ and $v\leq v(t)$ a.e.~in~$\Om$. For this reason,~\eqref{Intro1}\nobreakdash--\eqref{Intro4} could be referred to as \emph{local stability} properties, since they involve the local behavior of the energy functional~$\J$ close to the pair~$(u(t),v(t))$.

Finally, the energy-dissipation inequality~\eqref{energyInequ} is due to the lack of $1$-homogeneous term in the original Francfort-Marigo model. In~\cite{Knees2015}, the authors have been able to recover an energy balance for the phase-field model described by~\eqref{AmbTort2} thanks to a time reparametrization technique (see also~\cite{Mielke2012}). However, it resulted to be difficult to adapt such a trick to our finite-dimensional to continuum limit.

Following the main steps of numerical schemes, in order to construct a quasi-static evolution satisfying~\eqref{irreversibility}\nobreakdash--\eqref{energyInequ} we first discretize the function space~$H^{1}(\Om)$ and define the discrete counterpart of the functional~$\J$. More precisely, for every value of the mesh parameter $h\in\mathbb{N}\setminus\{0\}$ we construct a uniform triangulation~$\T_{h}$ of~$\Om$ (see~\eqref{K+}--\eqref{triangulation} for more details), we define the finite-dimensional space
\begin{equation}\label{Intro5.1}
	\F_{h}\coloneq\{u\in H^{1}(\Om):\,\text{$u$ is affine on $K$ for every $K\in\T_{h}$}\}\,,
\end{equation}
and we set, for every $u,v\in\F_{h}$,
\begin{equation}\label{Intro6}
	\J_{h}(u,v)\coloneq\frac{1}{2}\int_{\Om} \bigl( \p_{h}(v^{2})+\eta \bigr) \abs{\nabla u}^{2} \,\di x+\frac{1}{2}\int_{\Om}|\nabla{v}|^{2}\,\di x+\frac{1}{2}\int_{\Om}\p_{h}\bigl( (1-v)^{2} \bigr) \,\di x\,,
\end{equation}
where $\p_{h}\colon C(\overline{\Om})\to\F_{h}$ is the Lagrangian interpolation operator. In particular, the choice of the uniform triangulation~$\T_{h}$ is done here to simplify the exposition. We mention that we could also consider more general triangulations, taking into account the standard requirements arising from interpolation estimates (for more details see, e.g.,~\cite{Quarteroni1994}).

In the finite-dimensional framework described above, for every~$h\in\mathbb{N}\setminus\{0\}$ we construct a \emph{finite-dimensional quasi-static evolution} driven by the energy functional~$\J_h$ in~\eqref{Intro6} and satisfying the discrete counterpart of conditions~\eqref{irreversibility}\nobreakdash--\eqref{energyInequ} (see Definition~\ref{quasi-static} and Theorem~\ref{thm1}). The algorithm used for this purpose is a fusion of the one developed in~\cite{Artina2016} together with the alternating minimization of~\cite{Bourdin2000}. In particular, besides the usual time-discretization procedure, typical in the study of many rate-independent processes (see, for instance,~\cite{Mielke2005}), at each time step $t^{k}_{i}\coloneq\tfrac{iT}{k}$, $k\in\mathbb{N}\setminus\{0\}$, $i\in\{1,\ldots,k\}$, we construct a critical point of the energy~$\J_{h}$ at time~$t^{k}_{i}$ by solving the incremental minimum problems
\begin{align}
	&\min\, \{ \J_h (u, v_{j-1}) : u\in \F_h, u=w(t^{k}_{i}) \text{ on } \partial \Omega\}\,,\label{Intro7} \\
	&\min \, \{ \J_h (u_{j}, v) : v\in \F_h, v \leq v_{j-1} \text{ in } \Om \}\,,\label{Intro8}
\end{align}
where we have set, as initial conditions, $u_{0} \coloneq u^{h}_{h}(t^{k}_{i-1})$ and $v_{0}\coloneq v^{k}_{h}(t^{k}_{i-1})$. Denoting by~$u_{j}$ and~$v_{j}$ the solutions to~\eqref{Intro7} and~\eqref{Intro8}, respectively, we show in Proposition~\ref{prop2} that the pair~$(u_{j},v_{j})$ converges in $\F_{h}\times\F_{h}$ to a critical point of~$\J_{h}$, which we denote by~$(u^{k}_{h}(t^{k}_{i}),v^{k}_{h}(t^{k}_{i}))$.

In the proof of Theorem~\ref{thm1} we exploit the above algorithm in order to obtain a finite-dimensional quasi-static evolution $(u_{h},v_{h})\colon[0,T]\to \F_{h}\times\F_{h}$ in the sense of Definition~\ref{quasi-static}.

The last step of our construction is then the passage to the limit as the triangulation~$\T_{h}$ becomes finer and finer, that is, as the mesh parameter~$h$ tends to~$+\infty$. This is indeed the subject of the proof of Theorem~\ref{mainthm}, where we show that any limit of a sequence of finite-dimensional quasi-static evolutions is a quasi-static evolution in the sense of~\eqref{irreversibility}\nobreakdash--\eqref{energyInequ} above (see also Definition~\ref{quasi-staticcontinuous}). From a numerical viewpoint, this shows that the numerical results, obtained through a sort of finite-dimensional implementation of the damage model~\eqref{AmbTort} and~\eqref{AmbTort2} and described here by the time-continuous finite-dimensional evolutions~$(u_{h}(t),v_{h}(t))$, are actually close to the ``theoretical'' quasi-static evolutions~$(u(t),v(t))$ given by~\eqref{irreversibility}\nobreakdash--\eqref{energyInequ}. Moreover, we notice that the method we exploit to prove Theorems~\ref{mainthm} and~\ref{thm1} is also suitable for applications and numerical simulations, which, in particular, will be performed in Section~\ref{Numerics}.

In conclusion, we stress once again that the problem of existence of a quasi-static evolution for the phase-field model~\eqref{AmbTort} has been already tackled in various papers (see, for instance,~\cite{Giacomini2005,Knees2015,Negri2014}). In particular, in~\cite{Knees2015} an existence result of quasi-static evolution for the damage model via critical points of the energy functional~\eqref{AmbTort} has been achieved using, in a space-continuous setting, an alternate minimization scheme similar to the one described above. In~\cite{Negri2014}, instead, the convergence scheme is based on a local minimization procedure with respect to the damage variable~$v$. In~\cite{Giacomini2005} the evolution problem has been addressed in the setting of global minimizers, giving particular emphasis to the connection between the notions of quasi-static evolution in the phase-field model and in the variational ``sharp interface'' model of fracture (see, e.g.,~\cite{Francfort1998}). In view of these previous works, what we claim is new in our paper is not the existence result itself, but rather the technique used to construct an evolution, which is based on the algorithm given by~\eqref{Intro7} and~\eqref{Intro8} and which has been frequently used in numerical implementations (see~\cite{Artina2015,Bourdin2000,Bourdin2007,Bourdin2007a}). 

\subsection*{Plan of the paper}
The paper is organized as follows: in Section~\ref{Discrete} we present the evolution problem in full details, giving the definition of quasi-static evolution for the phase-field model (see Definition~\ref{quasi-staticcontinuous}) and stating the main result (Theorem~\ref{mainthm}). Then, we start discussing our discretization algorithm and, eventually, in Section~\ref{auxpb} we discuss the alternate minimization scheme which is at the core of our approximation. In Sections~\ref{qsfindim} and~\ref{sectionevolution} we prove Theorems~\ref{thm1} and~\ref{mainthm}, respectively. Finally, in Section~\ref{Numerics} we present some numerical simulations which exploit the alternate minimization algorithm discussed in this paper.

\section{Setting of the problem}\label{Discrete}
In this section, we describe the problem setting and introduce the main notation of the paper. We first start with the space-continuous notion, and in the second part of the section we discuss the space-discrete setting.

\subsection*{Space-continuous setting}
As already mentioned in the Introduction, we are studying quasi-static evolutions in the framework of phase-field approximation of brittle fractures in elastic bodies (for more details see, e.g.,~\cite{Ambrosio1990a,Ambrosio1992,Fonseca2007,Francfort1998}). Since the aim of this paper is to show a new constructive approach to the evolution problem based on a space discretization procedure, in order to keep the notations as simple as possible we focus here on a two dimensional model. In particular, we consider as a reference configuration the unit square $\Om\coloneq(0,1)^{2}$ in~$\R^{2}$. We believe that this is not a serious restriction and also evolutions in three dimensions can be similarly approached.

Once some~$\eta>0$ is fixed, we define the \emph{phase-field stored elastic energy} as
\begin{equation}\label{elastic}
	\E(u,v) \coloneq \frac{1}{2}\int_{\Om} (v^2+\eta) \abs{\nabla u}^{2}\,\di x \,,
\end{equation}
where $u\in H^{1}(\Om)$ denotes the antiplanar displacement and $v\in H^{1}(\Om)$ stands for the phase-field (or damage) variable. In particular, from~\eqref{elastic} we deduce that the elastic behavior of~$\Om$ depends pointwise on how damaged the body is, and, due to the presence of the positive parameter~$\eta$, the damage is never complete, in the sense that the elastic body~$\Om$ is always able to store a positive amount of elastic energy depending on the displacement~$u$. We also recall that the phase-field~$v$ is usually constrained to take values in the interval~$[0,1]$, where, for $x\in\Om$, $v(x)=0$ means that the elastic body~$\Om$ is experiencing a maximal damage in~$x$, while $v(x)=1$ means that the material is perfectly sound at~$x$. In order to avoid some technical issues related to the discrete setting described in the second part of this section, we simply assume~$v$ to belong to~$H^{1}(\Om)$. We will see how the above constraint can be naturally enforced in the space-discrete approximation of the evolution problem. We refer to Proposition~\ref{Prop1} for more details.

As usual in the phase-field approximation, we add to the stored elastic energy~\eqref{elastic} a dissipative term~$\D(v)$ which depends only on the damage $v\in H^{1}(\Om)$, namely,
\begin{equation}\label{D}
	\D(v) \coloneq \frac{1}{2} \int_{\Om} \bigr(\abs{\nabla v}^{2} + (1-v)^{2}\bigl) \,\di x \,.
\end{equation}
In the sense of $\Gamma$-convergence,  the dissipation functional $\D$ approximates, in the language of fracture mechanics, the energy dissipated by the crack production, as it has been shown in~\cite{Ambrosio1990a,Ambrosio1992}. 
 
We are now in a position to introduce the \emph{total phase-field energy} of the system as the sum of~\eqref{elastic} and~\eqref{D}: for every $u,v\in H^{1}(\Om)$, we simply set
\begin{equation}\label{J}
	\J(u,v)\coloneq \E(u,v) + \D(v) \,.
\end{equation}

As usual, the evolution problem will be driven by a time-dependent forcing term. In this case, given a time horizon $T>0$, we assume that the elastic body $\Om$ is subject to a Dirichlet boundary datum $w\in W^{1,2}([0,T]; H^{1}(\Om))$, so that, for every $t\in[0,T]$, the set of admissible displacement $\A(w(t))$ is defined by
\begin{equation}\label{A}
	\A(w(t)) \coloneq  \{u \in H^1 (\Om) : u=w(t) \text{ on } \partial\Om \} \,,
\end{equation}
where the equality has to be intended in the trace sense. The notation~\eqref{A} will be adopted also for functions $w\in H^{1}(\Om)$ not depending on time. 

In this context, a quasi-static evolution for the damage model is expressed by a pair displacement-damage $(u,v)\colon[0,T]\to H^{1}(\Om)\times H^{1}(\Om; [0,1])$. The first natural condition we want to impose is the so-called \emph{irreversibility} of the phase-field variable. Namely, the function $t\mapsto v(t)$ has to be non-increasing. This means that once the elastic body~$\Om$ is damaged, it can not be repaired, not even partially.

The second property a quasi-static evolution has to satisfy is a \emph{stability} condition. In our case, to be stable at time~$t$ means that the pair~$(u(t),v(t))$ is a critical point of the energy~\eqref{J} in the class of pairs $(u,v) \in H^1(\Omega) \times H^1(\Omega;[0,1])$ with $u\in \A (w(t))$ and $v\leq v(t)$ a.e.~in~$\Omega$. Since~$\J$ is Fréchet differentiable on~$H^1(\Om) \times ( H^1(\Om) \cap L^\infty (\Om) )$ (see~\cite[Proposition~1.1]{Burke2010}) with
\begin{align}
	\partial_{(u,v)}\J(u,v)[\varphi, \psi] = {}& \int_{\Om} ( v^{2} + \eta) \nabla u \cdot \nabla \varphi \,\di x + \int_{\Om} v \psi \abs{\nabla u}^{2} \,\di x \nonumber \\
	& + \int_{\Om} \nabla v \cdot \nabla \psi \,\di x - \int_{\Om} (1 - v) \psi \,\di x\,, \label{FrechetDeriv}
\end{align}
for every $u\in H^1(\Omega)$, $v\in H^1(\Omega) \cap L^\infty(\Omega)$, $\varphi \in C_c^\infty(\Omega)$, $\psi\in C^\infty(\overline{\Om})$, the stability condition can be written as
\begin{equation}\label{Frechet2}
	\partial_{(u,v)}\J\bigl( u(t),v(t) \bigr)[\varphi, \psi]\geq 0
\end{equation}
for every $\varphi\in C^{\infty}_{c}(\Om)$ and every $\psi\in C^{\infty}(\overline{\Om})$ with $\psi\leq0$.

\begin{remark}
	We notice that the inequality in~\eqref{Frechet2} and the restriction to test functions~$\psi \leq 0$ arise from the irreversibility condition of the damage variable~$v(t)$ discussed above.
\end{remark}

By the structure of the derivative of~$\J$~\eqref{FrechetDeriv}, inequality~\eqref{Frechet2} can be simply rephrased in terms of the following inequalities:
\begin{align}
	0 &= \partial_{u}\J \bigl( u(t),v(t) \bigr) [\varphi] = \int_{\Om} \bigl(  v^{2}(t) + \eta \bigr) \nabla  u(t)  \cdot \nabla{\varphi}\,\di x \label{CritPointu} \\
	0 &\leq \partial_{v}\J \bigl( u(t),v(t) \bigr) [\psi] \nonumber\\
	&= \int_{\Om}  v(t) \psi \bigabs{\nabla  u(t)}^{2} \,\di x + \int_{\Om} \nabla  v(t) \cdot \nabla \psi \,\di x - \int_{\Om} \bigl( 1 -  v(t) \bigr) \psi \,\di x \,, \label{CritPointv}
\end{align}
 for every $\varphi \in C^{\infty}_{c}(\Om)$ and $\psi\in C^{\infty}(\overline{\Om})$ with $\psi\leq0$.

\begin{remark}
	The right-hand sides of~\eqref{CritPointu} and~\eqref{CritPointv} represent the Gateaux derivatives in the direction of $u$ and $v$, respectively.

	We also notice that once we know that inequalities~\eqref{CritPointu} and~\eqref{CritPointv} are satisfied for every test functions $\varphi\in C^{\infty}_{c}(\Om)$ and $\psi\in C^{\infty}(\overline{\Om})$ with $\psi\leq0$ in~$\Om$, by density and truncation argument it is easy to see that they hold also for $\varphi\in H^{1}_{0}(\Om)$ and $\psi\in H^{1}(\Om)$, $\psi\leq0$ a.e. in~$\Om$.
\end{remark}

Finally, by the separate convexity of~$\J$ with respect to the variables~$u$ and~$v$, from formulas~\eqref{CritPointu} and~\eqref{CritPointv} we derive the actual stability condition, given in terms of minimum problems: for every~$t\in[0,T]$,~$u(t)$ minimizes~$\J(\cdot,v(t))$ in the class~$\A(w(t))$, while~$v(t)$ minimizes~$\J(u(t),\cdot)$ in the class of functions~$v\in H^{1}(\Om)$ such that $v\leq v(t)$ a.e.~in~$\Om$.

This leads us to the following definition of quasi-static evolution for the phase-field model via critical points of the energy~$\J$ in~\eqref{elastic}--\eqref{J}.
\begin{definition}\label{quasi-staticcontinuous}
	Let $T>0$ and $w\in W^{1,2}([0,T]; H^{1}(\Om))$. We say that a pair $( u, v)\colon [0,T]\to H^{1}(\Om) \times H^{1}(\Om)$ is a \emph{quasi-static evolution (of critical points)} if the following conditions are satisfied:
	\begin{enumerate}
		\item \emph{Irreversibility}: $0\leq v(t)\leq v(s)\leq1$ a.e.~in~$\Om$ for every $0\leq s\leq t\leq T$;
		
		\item \emph{Stability}: for every $t\in[0,T]$
		\begin{alignat}{2}
			\J \bigl( u(t), v(t) \bigr) &\leq \J \bigl( u, v(t) \bigr) \quad && \tforall u\in \A\bigl(w(t) \bigr), \label{contstabu}\\
			\J \bigl( u(t), v(t)\bigr) &\leq \J \bigl( u(t), v \bigr) \quad &&\tforall v\in H^{1}(\Om), v\leq v(t) \text{ a.e. in } \Om\,. \label{contstabv}
		\end{alignat}
		
		\item \emph{Energy-dissipation inequality}: for every $t\in[0,T]$
		\begin{equation}\label{energyineq2}
			\J \bigl(u(t),v(t) \bigr) \leq \J(u(0),v(0)) + \int_{0}^{t} \int_{\Om} (v^{2}(s)+\eta) \nabla u(s) \cdot \nabla \dot{w}(s) \,\di x \,\di s\,.
		\end{equation}
	\end{enumerate}
\end{definition}

We can now state the main existence result of the paper, which will be proved in Section~\ref{sectionevolution}.
\begin{theorem}\label{mainthm}
	Let $T>0$, $w\in W^{1,2}([0,T]; H^{1}(\Om))$, and $u_{0},\,v_{0}\in H^{1}(\Om)$ be such that $u_{0}\in \A (w(0))$ and $0\leq v_{0}\leq 1$ a.e.~in~$\Om$. Assume that the pair $(u_{0},v_{0})$ satisfies the stability conditions at time $t=0$:
	\begin{align}
		\J(u_{0},v_{0}) &\leq \J(u,v_{0}) \quad \tforall u\in \A(w(0)), \label{stabu0}\\
		\J(u_{0},v_{0}) &\leq \J(u_{0},v) \quad \tforall v\in H^{1}(\Om), \text{ such that } v\leq v_{0} \text{ a.e. in } \Om. \label{stabv0}
	\end{align}
	Then, there exists a quasi-static evolution $(u, v) \colon [0,T]\to H^{1}(\Om) \times H^{1}(\Om)$ with $u(0)=u_{0}$ and $v(0)=v_{0}$.
\end{theorem}

\begin{remark}\label{rmk2}
	We again stress that the study of existence of quasi-static evolution for the phase-field model based on the Ambrosio-Tortorelli functional~\eqref{AmbTort} is not a novelty. For instance, such a problem has been tackled in~\cite{Giacomini2005,Knees2015,Negri2014} using different convergence schemes, always in a continuous-space setting.
	
	What we claim is the main contribution of our paper is the technique used to construct such an evolution, which is based on the algorithm introduced in~\cite{Artina2015}. In particular, the method we exploit here is suitable for applications and numerical simulations, since, as we explain in the second part of this section, we first show the existence of a finite-dimensional quasi-static evolution in a discretized $H^{1}$-space (see~\eqref{discretespace}), and then pass to the limit as the discretization becomes finer and finer. We show that in the limit we recover a quasi-static evolution in the sense of Definition~\ref{quasi-staticcontinuous}. By this argument we deduce that the numerical scheme used to construct approximate solutions for the evolution problem is stable, i.e., it guarantees convergence to a suitable quasi-static evolution.
\end{remark}

\subsection{Space-discrete setting}
Let us now describe the space-discrete counterpart of the above setting. We first want to discretize the domain~$\Om$ following the basic ideas of the finite element method (for more details on the theory see, e.g., \cite{Quarteroni1994}). Therefore, given the mesh parameter $h\in\mathbb{N}{\setminus}\{0\}$, we define the uniform triangulation $\T_{h}$ of $\Om$ as follows: for every $k,j\in\{1,\ldots,h\}$ we set
\begin{align}
	K^{+}_{k,j} &\coloneq \Bigl\{ x = (x_1 , x_2) \in \Om : x_1 \in \Bigl[ \tfrac{k-1}{h}, \tfrac{k}{h} \Bigr], x_2 \in \Bigl[ \tfrac{j-1}{h}, \tfrac{j-1}{h} + \bigl( \tfrac{k}{h} - x_1 \bigr) \Bigr] \Bigr\}\,,\label{K+} \\
	K^{-}_{k,j} &\coloneq \Bigl\{ x = (x_1 , x_2) \in \Om : x_1 \in \Bigl[ \tfrac{k-1}{h}, \tfrac{k}{h} \Bigr], x_2 \in \Bigl[ \tfrac{j}{h} + \bigl( \tfrac{k-1}{h} - x_1 \bigr) ,\tfrac{j}{h} \Bigr] \Bigr\}\,,\label{K-}
\end{align}
so that $\T_{h}$ is the union of $K_{k,j}^{+}$ and $K_{k,j}^{-}$, namely,
\begin{equation}\label{triangulation}
	\T_{h}\coloneq \bigcup_{j,k=1}^{h} (K^{+}_{k,j} \cup K_{k,j}^{-}) \,.
\end{equation} 
Furthermore, we denote by~$\Delta_{h}$ the set of all the vertices of~$\T_{h}$ and we set $N_{h}\coloneq \#\Delta_{h}$. As we have already mentioned in the Introduction, the choice of the uniform triangulation~\eqref{K+}-\eqref{triangulation} is made in order to simplify the exposition. More general triangulation~$\T_{h}$ can be considered, fulfilling the usual bound for interpolation estimates. We refer to~\cite{Quarteroni1994} for more details.

Once we are given the triangulation~$\T_{h}$, we need to discretize the function space~$H^{1}(\Om)$. Therefore, we define the finite-dimensional function space~$\F_{h}$ as the set of continuous functions on~$\overline{\Om}$ that are affine on each triangle $K\in \T_h$. More precisely, we set
\begin{equation}\label{discretespace}
	\F_{h} \coloneq \{u \in C(\overline{\Om}) \cap H^{1}(\Om) : \nabla u \text{ is constant a.e.~on~$K$ for every $K\in\T_{h}$}\}\,.
\end{equation}

Clearly,~$\F_{h}$ can be endowed with the usual~$H^{1}$-norm. Moreover, we fix a basis~$\{\xi_{l}\}_{l=1}^{N_{h}}$ of~$\F_{h}$ in the following natural way: for every~$l=1,\ldots, N_{h}$, the element~$\xi_{l}\in\F_{h}$ is such that
\begin{equation}\label{basis}
	\xi_{l}(x_{m})=\delta_{lm} \qquad\text{for every $x_{m}\in\Delta_{h}$},
\end{equation}
where~$\delta_{lm}$ is the Kronecker delta. We further assume that the basis~$\{\xi_{l}\}_{l=1}^{N_{h}}$ satisfies the stiffness condition
\begin{equation}\label{stiffness}
	\int_{\Om}\nabla \xi_{l} \cdot \nabla \xi_{m} \,\di x \leq 0 \qquad \text{for every $l,m\in\{1,\dotsc,N_{h}\}$, $l\neq m$},
\end{equation}
which is fulfilled, e.g., if the angles of the triangles are smaller or equal to~$\frac{\pi}{2}$ (see~\cite{Ciarlet1973}).

In this framework, we introduce the discrete counterpart of the stored elastic energy~\eqref{elastic} and of the dissipated energy~\eqref{D}: for every $u,v\in\F_{h}$, we set
\begin{align}
	\E_{h}(u,v) &\coloneq \frac{1}{2}\int_{\Om} \bigl( \p_{h}(v^{2})+\eta \bigr) \abs{\nabla u}^{2} \,\di x \,,\label{discreteelastic}\\
	\D_{h}(v) &\coloneq \frac{1}{2}\int_{\Om}|\nabla{v}|^{2}\,\di x+\frac{1}{2}\int_{\Om}\p_{h}\bigl( (1-v)^{2} \bigr) \,\di x\,, \label{discreteD}
\end{align}
where $\p_{h}\colon C(\overline{\Om})\to \F_{h}$ is the Lagrangian interpolant onto the space~$\F_{h}$, i.e., the unique operator defined on $C(\overline{\Om})$ with values in~$\F_{h}$ such that
\begin{equation}\label{ph}
	\p_{h}(\varphi)(x_{l}) = \varphi(x_{l}) \qquad \text{for every $\varphi\in C(\overline{\Om})$ and every $x_{l}\in\Delta_{h}$}.
\end{equation}
As in~\eqref{J}, the \emph{discrete total energy} is the sum of $\E_{h}$ and $\D_{h}$. Hence,
\begin{equation}\label{discretefunctional}
	\J_{h}(u,v) \coloneq \E_{h}(u,v) + \D_{h}(v)\,.
\end{equation}

We note that, thanks to~\cite{Quarteroni1994}, we can also approximate the Dirichlet boundary datum in~$\F_{h}$. More precisely, there exists a sequence $w_{h} \in  W^{1,2} ([0,T]; \F_{h})$ such that $w_{h}\to w$ in~$W^{1,2} ([0,T]; H^{1}(\Om))$ as $h \to +\infty$. In particular, this implies that $w_{h}(t)\to w(t)$ in~$H^{1}(\Om)$ for every $t\in[0,T]$ and $\dot{w}_{h}(t)\to\dot{w}(t)$ in~$H^{1}(\Om)$ for a.e. $t\in[0,T]$. Hence, the quasi-static evolution in the space-discrete setting (see Definition~\ref{quasi-static}) will be driven by the approximate boundary datum $w_{h}$, and, as in~\eqref{A}, for every~$h$ and every $t\in[0,T]$ we restrict the set of admissible displacements to
\begin{equation}\label{discreteA}
	\A_{h}(w_{h}(t)) \coloneq \{u\in\F_{h} :  u=w_{h}(t) \text{ on } \partial\Om \}\,.
\end{equation}

Analogously to Definition~\ref{quasi-staticcontinuous}, the notion of finite-dimensional quasi-static evolution reads as follows:
\begin{definition}\label{quasi-static}
	Let $T>0$ and $h\in\mathbb{N}{\setminus}\{0\}$ be fixed. Let $w_{h}\in W^{1,2}([0,T];\F_{h})$. We say that a pair of functions $(u_{h},v_{h})\colon[0,T]\to\F_{h}{\,\times\,}\F_{h}$ is a \emph{finite-dimensional quasi-static evolution} if it satisfies the following conditions:
	\begin{enumerate}
		\item \label{irreversCond} \emph{Irreversibility}: $0\leq v_{h}(t)\leq v_{h}(s)\leq 1$ in~$\Om$ for every $0\leq s\leq t\leq T$;
		\item \label{stabilityCond}\emph{Stability}: for every $t\in(0,T]$ we have
		\begin{align}
			\J_{h} \bigl( u_{h}(t), v_{h}(t) \bigr) &\leq \J_{h} \bigl( u, v_{h}(t) \bigr) \qquad \text{for every $u\in\A_{h}(w_{h}(t))$} \,, \label{stab1}\\
			\J_{h} \bigl( u_{h}(t), v_{h}(t) \bigr) &\leq \J_{h} \bigl( u_{h}(t), v \bigr) \qquad \text{for every $v\in\F_{h}, v\leq v_{h}(t)$ in~$\Om$} \,; \label{stab2}
		\end{align} 
		\item \emph{Energy-dissipation inequality}: for every $t\in[0,T]$
		\begin{align}
			\J_{h}\bigl( u_{h}(t),v_{h}(t) \bigr) \leq{}& \J_{h} \bigl( u_{h}(0), v_{h}(0) \bigr) \nonumber \\
			&+ \int_{0}^{t} \int_{\Om} \bigl( \p_{h} (v_{h}^{2}(s)) + \eta \bigr) \nabla u_{h}(s) \cdot \nabla \dot{w}_{h}(s) \,\di x \,\di s\,. \label{energyineq}
		\end{align}
	\end{enumerate}
\end{definition}

\begin{remark}
Let us briefly comment on the stability condition~\eqref{stabilityCond} of Definition~\ref{quasi-static}. As in the space-continuous setting, being the functional~$\J_{h}$ separately convex with respect to the variables~$u$ and~$v$, inequalities~\eqref{stab1} and~\eqref{stab2} are equivalent to
\begin{align*}
	0 &=\partial_{u} \J_{h}(u_{h}(t),v_{h}(t)) = \int_{\Om} (\p_{h}(v_{h}^{2}(t)) + \eta) \nabla{u_{h}(t)} \cdot \nabla{\varphi} \,\di x \,, \\
	0&\leq \partial_{v}\J_{h}(u_{h}(t),v_{h}(t))\\
	&= \int_{\Om} \p_{h}(v_{h}(t)\psi) \abs{\nabla{u_{h}(t)}}^{2} \, \di x + \int_{\Om} \nabla{v_{h}(t)} \cdot \nabla{\psi} \,\di x - \int_{\Om} \p_{h} \Bigl( \bigl(1-v_{h}(t) \bigr)\psi\Bigr)\,\di x \,,
\end{align*}
 for every $\varphi\in\F_{h}\cap H^{1}_{0}(\Om)$ and $\psi\in\F_{h}$ with $\psi\leq0$ in~$\Om$.
 
 Moreover, we want the property~\eqref{stabilityCond} to be satisfied only in the interval~$(0,T]$. The motivation of this choice is the following: in Theorem~\ref{mainthm} (see also Section~\ref{sectionevolution}) we aim to construct a quasi-static evolution in the space-continuous setting as limit of finite-dimensional quasi-static evolutions. For this reason, as it will be shown in the proof of Theorem~\ref{mainthm}, we need to find \emph{ad hoc} approximations of the initial conditions $u_{0},\,v_{0}$ in the space~$\F_{h}$. In doing this, we can not guarantee to keep track of the stability properties~\eqref{stabu0}\nobreakdash--\eqref{stabv0} of the pair~$(u_{0},v_{0})$ (see Theorem~\ref{mainthm}). Therefore, at this stage it is enough for us to have stability for strictly positive time, while in the space-continuous limit we will recover it also for $t=0$. 
\end{remark}

In the following theorem, we state the existence of a finite-dimensional quasi-static evolution. Before stating the proof, we need some auxiliary results. Here, we only mention that the construction of a finite-dimensional quasi-static evolution is based on the incremental procedure described in Section~\ref{auxpb} and presented in a different context in~\cite{Artina2016}.

\begin{theorem}\label{thm1}
	Let $h\in\mathbb{N}{\setminus}\{0\}$, $w_{h}\in W^{1,2}([0,T];\F_{h})$, and $u_{0,h},v_{0,h}\in\F_{h}$ be such that $u_{0,h}\in\A_{h}(w_{h}(0))$ and $v_{0,h}\geq0$ in~$\Om$. Then, there exists a finite dimensional quasi-static evolution~$(u_{h},v_{h})\colon[0,T]\to\F_{h}{\,\times\,}\F_{h}$ such that $(u_{h}(0),v_{h}(0))=(u_{0,h},v_{0,h})$.
\end{theorem}

In the next section, we describe the core of our convergence algorithm, which will allow us to construct critical points of the energy functional~$\J_{h}$ satisfying proper boundary and irreversibility conditions. Such a scheme will be exploited in the proof of Theorem~\ref{thm1} to find a suitable time-discrete approximation of a finite-dimensional quasi-static evolution.

\section{The alternate minimization scheme}\label{auxpb}
For a given mesh parameter~$h\in\mathbb{N} \setminus \{0\}$, let us fix two functions $v_{0}, w\in\F_{h}$, with~$v_{0}\geq0$ in~$\Om$. In what follows, we show a constructive way to find a critical point $(\bar{u},\bar{v}) \in \F_{h} \times \F_{h}$ of the functional~$\J_{h}$ under the constraints $\bar u = w$ on~$\partial \Omega$ and $\bar v \leq v_0$ in~$\Om$. 

The recursive scheme we adopt here is a modification of the strategy presented in~\cite{Artina2015} in a finite-dimensional setting, and similar to the alternate minimization procedure used in~\cite{Knees2015} for the ``continuum" phase-field model. What we have to take care of in our context is the presence of the operator~$\p_{h}$ of~\eqref{ph} in the definition of the functional~$\J_{h}$ (see~\eqref{discreteelastic}--\eqref{discretefunctional}). 

For every $j\in\mathbb{N}{\setminus}\{0\}$ we define the functions~$u_{j}$ and~$v_{j}$ in~$\F_{h}$ as follows:
\begin{align}
	u_{j} &\coloneq \argmin \,\{\J_{h}(u,v_{j-1}) :\, u\in\A_{h}(w) \} \,, \label{minu}\\
	v_{j} &\coloneq \argmin \,\{\J_{h}(u_{j},v) : v \in \F_{h},\, v\leq v_{j-1} \text{ in~$\Om$} \} \,. \label{minv}
\end{align} 
The existence of minimizers of~\eqref{minu} is standard.  The uniqueness follows by the strict convexity of the functional~$\J_{h}(\cdot, v)$ for~$v\in\F_{h}$.

In the following proposition, we briefly discuss the existence and uniqueness of~$v_{j}$. We also show the usual bound $0\leq v_{j}\leq1$, which does not follow by simple truncation argument because of the nature of the function space $\F_{h}$ and of the presence of the interpolation operator~$\p_{h}\colon C(\overline{\Om})\to\F_{h}$. The proof is contained in the~\ref{ProofProp1}.

\begin{proposition}\label{Prop1}
	The minimum problem~\eqref{minv} admits a unique solution. Moreover, the solution~$v_{j}\in\F_{h}$ satisfies $0\leq v_{j}\leq 1$ for every $j\in\mathbb{N}{\setminus}\{0\}$.
\end{proposition}

\begin{proof}
	See \ref{ProofProp1}.
\end{proof}

In what remains of this section, we show that any limit $(\bar{u},\bar{v})$ of the sequence $(u_{j},v_{j})$ defined in~\eqref{minu}\nobreakdash--\eqref{minv} is a critical point of the functional~$\J_{h}$ satisfying $\bar{u}\in\A_{h}(w)$ and $\bar{v} \leq v_{0}$ in~$\Om$. To do so, we first show a stability property of the minimum problems~\eqref{minu} and~\eqref{minv}. This is the aim of the following lemma, which is stated in a more general setting than the one needed in this section, since it will be useful also in the proof of Theorem~\ref{thm1}.

\begin{lemma}\label{lemma1}
	Let $u_{k}, v_{k}, w_{k}, z_{k} \in \F_{h}$ be such that $u_{k}\in\A_{h}(w_{k})$ and
	\begin{alignat}{2}
		\J_{h}(u_{k},z_{k}) &\leq \J_{h}(u,z_{k}) &&\qquad \text{for every $u\in\A_{h}(w_{k})$}, \label{a}\\
		\J_{h}(u_{k},v_{k}) &\leq \J_{h}(u_{k},v) &&\qquad \text{for every $v\in\F_{h}$ such that $v\leq z_{k}$ in~$\Om$}. \label{b}
	\end{alignat} 
	Assume that there exist $\bar{u}, \bar{v}, \bar{w}, \bar{z}\in\F_{h}$ such that $u_{k}\to \bar{u}$, $v_{k}\to \bar{v}$, $w_{k}\to \bar{w}$, and $z_{k}\to \bar{z}$ in~$\F_{h}$ as $k\to+\infty$. Then $\bar{u}\in\A_{h}(\bar{w})$ and
	\begin{alignat}{2}
		\J_{h}(\bar{u},\bar{z}) &\leq \J_{h}(u,\bar{z}) &&\qquad \text{for every $u\in\A_{h}(\bar{w})$},\label{c}\\
		\J_{h}(\bar{u},\bar{v}) &\leq \J_{h}(\bar{u},v) &&\qquad \text{for every $v\in\F_{h}$ such that $v\leq \bar{z}$ in~$\Om$} \,.\label{d}
	\end{alignat}
\end{lemma}

\begin{proof}
	Let us prove~\eqref{c}. For every $k\in\mathbb{N}$ and every $u\in\A_{h}(\bar{w})$ we have
	\begin{equation}\label{0.1}
		\J_{h}(u_{k},z_{k})\leq\J_{h}(u+w_{k}-\bar{w},z_{k})\,.
	\end{equation}
	Since $u+w_{k}-\bar{w}\to u$ in~$\F_{h}$ as $k\to+\infty$, passing to the limit in~\eqref{0.1} we get~\eqref{c}.
	
	As for~\eqref{d}, for every $v\in\F_{h}$ such that $v\leq\bar{z}$ in~$\Om$ we have that $z_{k}+v-\bar{z}\leq z_{k}$. Hence, by~\eqref{b},
	\begin{equation*}
		\J_{h}(u_{k},v_{k})\leq\J_{h}(u_{k},z_{k}+v-\bar{z})\,.
	\end{equation*}
	Passing to the limit in the previous inequality we get~\eqref{d}.
\end{proof}

We are now ready to show the convergence of the sequence $(u_{j},v_{j})$, defined by~\eqref{minu} and~\eqref{minv}, to a critical point of~$\J_{h}$.
\begin{proposition}\label{prop2}
	Let $v_0, w \in \F_{h}$ with $v_0 \geq 0$, and let $u_j, v_j$ be defined by~\eqref{minu} and~\eqref{minv}, respectively. Then the following facts hold:
	\begin{enumerate}
		\item there exist $\bar{u}, \bar{v}\in\F_{h}$ such that $u_{j}\to \bar{u}$ and $v_{j}\to \bar{v}$ in~$\F_{h}$ as $j\to+\infty$; \label{convergence}
		\item the limit function $\bar v$ satisfies $0\leq \bar v \leq 1$; \label{3}
		\item \label{stability} the limit functions $\bar{u}, \bar{v}\in\F_{h}$ satisfy
		\begin{alignat}{2}
			\J_{h}(\bar{u},\bar{v}) &\leq \J_{h}(u,\bar{v}) \qquad &&\text{for every $u\in\A_{h}(w)$} \,, \label{1}\\
			\J_{h}(\bar{u},\bar{v}) &\leq \J_{h}(\bar{u},v) \qquad &&\text{for every $v\in\F_{h}$ with $v\leq \bar{v}$} \,. \label{2}
		\end{alignat}
	\end{enumerate}
\end{proposition}

\begin{proof}
	By definition of~$u_{j}$ and $v_{j}$, for every $j\geq2$ we have
	\begin{equation}\label{4}
		\J_{h}(u_{j},v_{j})\leq\J_{h}(u_{j},v_{j-1})\leq\J_{h}(u_{j-1},v_{j-1})\,.
	\end{equation}
	Iterating inequality~\eqref{4}, we obtain 
	\begin{equation*}
		\J_{h}(u_{j},v_{j})\leq\J_{h}(u_{1},v_{1})\leq\J_{h}(w,v_{0})<+\infty\,,
	\end{equation*}
	from which we deduce that the sequences~$u_{j}$ and~$v_{j}$ are bounded in~$\F_{h}$. Being $v_{j}$ a decreasing sequence with values in the interval~$[0,1]$ and being~$\F_{h}$ finite-dimensional, we deduce that there exists~$\bar{v}\in\F_{h}$ such that $v_{j}\to\bar{v}$ in~$\F_{h}$ and $0\leq\bar{v}\leq1$ in~$\Om$, so that property~\eqref{3} holds. Moreover, by compactness, there exists~$\bar{u}\in\F_{h}$ such that, up to a subsequence, $u_{j}\to\bar{u}$.	
	
	Property~\eqref{stability} results from Lemma~\ref{lemma1} applied to the sequences $u_{j}$, $v_{j}$, with fixed boundary datum $w\in\F_{h}$. By uniqueness of solution to~\eqref{1}, we also deduce that the whole sequence~$u_{j}$ converges to~$\bar{u}$ in~$\F_{h}$, and the proof is thus concluded.
\end{proof}

\section{Construction of time-continuous finite-dimensional quasi-static evolutions}\label{qsfindim}
We are now ready to prove the existence of a finite-dimensional quasi-static evolution in the sense of Definition~\ref{quasi-static}. The strategy of the proof is based on a time discretization procedure, typical of many rate-independent processes (see, e.g.,~\cite{Mielke2005}): we approximate continuous in time evolutions with discrete in time limits of the auxiliary incremental problems discussed in Section~\ref{auxpb}. As the time step tends to zero, we obtain a finite-dimensional quasi-static evolution in the sense of Definition~\ref{quasi-static}.

\begin{proof}[Proof of Theorem \ref{thm1}.]
	We perform in details the time discretization procedure described above for fixed $h\in\mathbb{N}{\setminus}\{0\}$ and given initial condition $v_{0,h}$ and $u_{0,h}$ as in the statement of the theorem.
	
	 For every $k\in\mathbb{N}$, we consider the uniform subdivision of the time interval~$[0,T]$ given by $t^{k}_{i}\coloneq \frac{iT}{k}$, $i=0,\ldots,k$. In order to construct a discrete in time evolution in the finite-dimensional setting we follow the algorithm proposed in~\cite{Artina2015}: for $i=0$ we set $u^{k,h}_{0}\coloneq u_{0,h}$ and $v^{k,h}_{0}\coloneq v_{0,h}$. For~$i\geq1$, at the instant~$t^{k}_{i}$ 
	we construct a critical point $(u^{k,h}_{i},v^{k,h}_{i})\in\F_{h}{\,\times\,}\F_{h}$ of $\J_{h}$ as limit of the alternating minimization process described in Section~\ref{auxpb} with initial condition $(u^{k,h}_{i-1},v^{k,h}_{i-1})$. More precisely, let us set $u^{k,h}_{i,0}\coloneq u^{k,h}_{i-1}$ and $v^{k,h}_{i,0}\coloneq  v^{k,h}_{i-1}$. For $j\in\mathbb{N}$, $j\geq1$, we define iteratively two sequences of functions~$u^{k,h}_{i,j}$ and~$v^{k,h}_{i,j}$ as
	\begin{align}
		u^{k,h}_{i,j} &\coloneq \argmin \{\J_{h}(u,v^{k,h}_{i,j-1}) : u\in\A_{h}(w_{h}(t^{k}_{i}))\} \,, \label{10}\\
		v^{k,h}_{i,j} &\coloneq \argmin \{\J_{h}(u^{k,h}_{i,j},v) : v\in\F_{h} ,\, v\leq v^{k,h}_{i,j-1}\} \,. \label{11}
	\end{align}
	We notice that, since by assumption $v^{k,h}_{0}\geq0$, combining Propositions~\ref{Prop1} and~\ref{prop2} we deduce that~\eqref{10} and~\eqref{11} always admit unique solutions and, for every $k\in\mathbb{N}$ and every $i\in\{1,\ldots,k\}$, there exist $u^{k,h}_{i}, v^{k,h}_{i}\in\F_{h}$ such that $u^{k,h}_{i,j}\to u^{k,h}_{i}$ and $v^{k,h}_{i,j}\to v^{k,h}_{i}$ in~$\F_{h}$ as $j\to+\infty$. Moreover, $0\leq v^{k,h}_{i}\leq v^{k,h}_{i-1}\leq1$ in~$\Om$ and, again thanks to Proposition~\ref{prop2},
	\begin{alignat}{2}
		\J_{h}(u^{k,h}_{i},v^{k,h}_{i}) &\leq\J_{h}(u,v^{k,h}_{i}) &&\qquad \text{for every $u\in\A_{h}(w_{h}(t^{k}_{i}))$},\label{12}\\
		\J_{h}(u^{k,h}_{i},v^{k,h}_{i}) &\leq\J_{h}(u^{k,h}_{i},v) &&\qquad \text{for every $v\in\F_{h},\,v\leq v^{k}_{i}$ in $\Om$}.\label{13}
	\end{alignat}
	
	We now introduce the following piecewise constant interpolation functions: for every $k\in\mathbb{N}$ and every $i\in\{1,\ldots,k\}$ we set
	\begin{equation}\label{interpolant}
		 u^{k}_{h}(t) \coloneq u^{k,h}_{i-1} \qquad \text{and} \qquad v^{k}_{h}(t)\coloneq v^{k,h}_{i-1} \qquad \text{for every $t\in[t^{k}_{i-1},t^{k}_{i})$}.
	 \end{equation}
	
	We now show a discrete energy inequality for the sequence $(u^{k}_{h}(\cdot),v^{k}_{h}(\cdot))$. Namely, we prove that for every $k\in\mathbb{N}$, every $i\in\{1,\ldots,k\}$, and every $t\in[t^{k}_{i-1},t^{k}_{i})$
	\begin{align}
	 	\J_{h}(u_{h}^{k}(t),v_{h}^{k}(t)) \leq{}& \J_{h}(u_{h}^{k}(0),v_{h}^{k}(0)) \nonumber\\
	 	& + \int_{0}^{t^{k}_{i}} \int_{\Om} \bigl(\p_{h}(v_{h}^{k}(s))^{2} + \eta \bigr) \nabla u_{h}^{k}(s) \cdot \nabla \dot{w}_{h}(s) \,\di x\,\di s+R_{k} \label{14}
	\end{align}
	for some rest~$R_{k}$ converging to~$0$ as $k\to+\infty$. 
	
	Let us fix $k\in\mathbb{N}$ and $i\in\{1,\ldots,k\}$. By the construction~\eqref{10} and~\eqref{11} of the sequences~$(u^{k}_{i,j})_{j\in\mathbb{N}}$ and~$(v^{k}_{i,j})_{j\in\mathbb{N}}$, for every $j\geq2$ we have
	\begin{equation}\label{4.1}
		\J_{h}(u^{k,h}_{i,j},v^{k,h}_{i,j})\leq\J_{h}(u^{k,h}_{i,j},v^{k,h}_{i,j-1})\leq\J_{h}(u^{k,h}_{i,j-1},v^{k,h}_{i,j-1})\,.
	\end{equation}
	Iterating inequality~\eqref{4.1}, we deduce that 
	\begin{equation}\label{5}
		\J_{h}(u^{k,h}_{i,j},v^{k,h}_{i,j})\leq\J_{h}(u^{k,h}_{i,1},v^{k,h}_{i,1})\,.
	\end{equation}
	
	Since $u^{k,h}_{i-1}+w_{h}(t^{k}_{i})-w_{h}(t^{k}_{i-1})\in\A_{h}(w_{h}(t^{k}_{i}))$, again by definition~\eqref{10} and~\eqref{11} of the pair $(u^{k,h}_{i,1},\,v^{k,h}_{i,1})$, we can continue in~\eqref{5} obtaining
	\begin{align}
			&\J_{h} (u^{k,h}_{i,j},v^{k,h}_{i,j}) \leq \J_{h}(u^{k,h}_{i,1},v^{k,h}_{i-1}) \nonumber \\
			\leq{}& \J_{h}  \bigl(u^{k,h}_{i-1}+w_{h}(t^{k}_{i})-w_{h}(t^{k}_{i-1}),v^{k,h}_{i-1} \bigr) \nonumber \\
			= {} &\J_{h}(u^{k,h}_{i-1},v^{k,h}_{i-1}) + \int_{\Om} \bigl( \p_{h}((v^{k,h}_{i-1})^{2}) + \eta \bigr) \nabla u^{k,h}_{i-1} \cdot \nabla \bigl( w_{h}(t^{k}_{i}) - w_{h}(t^{k}_{i-1}) \bigr) \di x \nonumber \\
			& + \frac{1}{2} \int_{\Om} \bigl( \p_{h}((v^{k,h}_{i-1})^{2}) + \eta \bigr) \bigabs{\nabla \bigl( w_{h}(t^{k}_{i}) - w_{h}(t^{k}_{i-1}) \bigr)}^{2} \,\di x\,. \label{6}
	\end{align}
	
	Recalling the definition~\eqref{interpolant} of the interpolant functions~$u_{h}^{k}(\cdot), v_{h}^{k}(\cdot)$, the hypothesis $w_{h}\in W^{1,2}([0,T];\F_{h})$, and the condition $0\leq v^{k,h}_{i-1}\leq1$, we can estimate~\eqref{6} by
	\begin{align}
		\J_{h}(u^{k,h}_{i,j},v^{k,h}_{i,j}) \leq{}& \J_{h}(u^{k,h}_{i-1},v^{k,h}_{i-1}) + \int_{t^{k}_{i-1}}^{t^{k}_{i}} \int_{\Om} \bigl( \p_{h}((v_{h}^{k}(s))^{2}) + \eta \bigr) \nabla u_{h}^{k}(s) \cdot \nabla \dot{w}_{h}(s) \,\di x\,\di s \nonumber \\
		&+ \frac{T(1 + \eta)}{2k} \int_{t^{k}_{i-1}}^{t^{k}_{i}} \bignorm{\dot{w}_{h}(s)}^{2}_{H^{1}} \,\di s\,.\label{15}
	\end{align}
	Passing to the limit in~\eqref{15} as $j\to+\infty$ we get
	\begin{align}
		\J_{h}(u^{k,h}_{i},v^{k,h}_{i}) \leq{}& \J_{h}(u^{k,h}_{i-1},v^{k,h}_{i-1}) + \int_{t^{k}_{i-1}}^{t^{k}_{i}} \int_{\Om} \bigl( \p_{h}((v_{h}^{k}(s))^{2}) + \eta \bigr) \nabla u_{h}^{k}(s) \cdot \nabla \dot{w}_{h}(s) \,\di x\,\di s \nonumber \\
		&+ \frac{T(1 + \eta)}{2k} \int_{t^{k}_{i-1}}^{t^{k}_{i}} \bignorm{\dot{w}_{h}(s)}^{2}_{H^{1}} \,\di s\,.\label{16}
	\end{align}
	Finally, iterating inequality~\eqref{16} and setting
	\begin{equation*}
		R_{k} \coloneq \frac{T(1 + \eta)}{2k}\int_{0}^{T} \bignorm{\dot{w}_{h}(s)}^{2}_{H^{1}}\,\di s\,,
	\end{equation*}
	we deduce~\eqref{14} for every $t\in[t^{k}_{i-1},t^{k}_{i})$.
	
	Let us now perform the passage to the time-continuous limit. In view of the discrete energy inequality~\eqref{14}, of the hypothesis $w_{h}\in W^{1,2}([0,T];\F_{h})$, and of the uniform bound $0\leq v_{h}^{k}(t)\leq 1$ for $k\in\mathbb{N}$ and $t\in[0,T]$, we deduce that
	\begin{equation}\label{17}
		\sup_{\substack{t \in [0,T] \\ k \in \mathbb{N}}} \norm{u_{h}^{k}(t)}_{H^{1}} < +\infty \qquad \text{and} \qquad \sup_{\substack{t \in [0,T] \\ k \in \mathbb{N}}} \norm{ v_{h}^{k}(t)}_{H^1} < +\infty\,.
	\end{equation}
	
	Since the functions $v_{h}^{k}\colon[0,T]\to\F_{h}$ are non-increasing in time and uniformly bounded, we can apply Helly's Selection Principle (see~\cite[Section~VIII.~3.~§4]{Natanson1964}). Consequently, there exists a non-increasing function $v_{h}\colon[0,T]\to\F_{h}$ such that, along a subsequence~$k_{l}$, $v_{h}^{k_{l}}(t)\to v_{h}(t)$ in~$\F_{h}$ for every $t\in [0,T]$. In particular,~$v_{h}$ satisfies the irreversibility condition~\eqref{irreversCond} of Definition~\ref{quasi-static}.
	 
	In view of the first inequality in~\eqref{17}, for every $t\in[0,T]$ there exists a function $u_{h}(t)\in\F_{h}$ such that, up to a subsequence, $u_{h}^{k_{l}}(t)\to u_{h}(t)$ in~$\F_{h}$. Applying Lemma~\ref{lemma1} we obtain that~$u_{h}(t)$ and~$v_{h}(t)$ satisfy the stability conditions~\eqref{stab1} and~\eqref{stab2} of Definition~\ref{quasi-static}. In particular, $u_h (t)$ being the unique solution of~\eqref{stab1}, we deduce that for every $t\in [0,T]$ the whole sequence~$u_{h}^{k_{l}}(t)$ converges to~$u_{h}(t)$ in~$\F_{h}$ as $k_{l}\to+\infty$.
	
	Finally, in order to prove the energy inequality~\eqref{energyineq} of Definition~\ref{quasi-static}, we only need to pass to the limit in the discrete energy inequality~\eqref{14}. Since all the convergences described above are strong and hold for every $t\in[0,T]$, letting $k_l\to+\infty$ we deduce~\eqref{energyineq}, and the proof is thus concluded.
\end{proof}

\section{From space-discrete to space-continuous evolution}\label{sectionevolution}
This section is devoted to the proof of Theorem~\ref{mainthm}. In particular, we show that any limit of a sequence $(u_{h},v_{h})\colon[0,T]\to\F_{h}{\,\times\,}\F_{h}$ of finite-dimensional quasi-static evolutions is a quasi-static evolution in the sense of Definition~\ref{quasi-staticcontinuous}. 

Before proving Theorem~\ref{mainthm}, we show two useful properties. Firstly, we state a uniform estimate on the operator~$\p_{h}$.  Secondly, we prove a stability property of the functionals~$\J_{h}$ and~$\J$ analogous to Lemma~\ref{lemma1}, and which takes into account the ``convergence'' of the finite dimensional spaces $\F_{h}$ to $H^{1}(\Om)$ as $h\to+\infty$. 

Although the following result is standard, we provide its proof in the \ref{ProofLemma0} for the sake of completeness.
\begin{lemma}\label{lemma0}
	Let $h\in\mathbb{N}{\setminus}\{0\}$, let $\p_{h}\colon C(\overline{\Om})\to H^{1}(\Om)$ be the operator defined by~\eqref{ph}, and let $g\in C^{2}(\R)$. Then, for every~$M>0$ there exists a positive constant~$C=C(g,M)$ depending only on~$g$ and~$M$ such that for every~$v\in\F_{h}$ with~$\|v\|_{\infty}\leq M$
	\begin{equation}\label{discreteestimate}
		\bignorm{(g\circ v) - \p_{h}( g\circ v)}_{1} \leq Ch^{-2} \norm{\nabla v}_{2}^{2} \,.
	\end{equation}
\end{lemma}

\begin{proof}
	See \ref{ProofLemma0}.	
\end{proof}

\begin{lemma}\label{lemma2}
	For every $h\in\mathbb{N}$, let $w_{h}, u_{h},v_{h}\in\F_{h}$ be such that $u_{h}\in\A_{h}(w_{h})$, $0\leq v_{h}\leq1$ in~$\Om$, and
	\begin{alignat}{2}
		\J_{h}(u_{h},v_{h}) &\leq \J_{h}(u,v_{h}) &&\qquad \text{for every $u\in\A_{h}(w_{h})$}, \label{20}\\
		\J_{h}(u_{h},v_{h}) &\leq \J_{h}(u_{h},v) &&\qquad \text{for every $v\in\F_{h}$ such that $v\leq v_{h}$}. \label{21}
	\end{alignat}
	Assume that $w_{h}\to\bar{w}$ in~$H^{1}(\Om)$, $u_{h}\rightharpoonup \bar{u}$ and $v_{h}\rightharpoonup \bar{v}$ weakly in~$H^{1}(\Om)$ as $h\to+\infty$. Then $\bar{u}\in\A(\bar{w})$ and
	\begin{alignat}{2}
		\J(\bar{u},\bar{v}) &\leq \J(u,\bar{v}) &&\qquad \text{for every $u\in \A(\bar{w})$}, \label{22} \\
		\J(\bar{u},\bar{v}) &\leq \J(\bar{u},v) &&\qquad \text{for every $v\in H^{1}(\Om)$ with $v\leq\bar{v}$}. \label{23}
	\end{alignat}
	Moreover, $u_{h}\to\bar{u}$ strongly in~$H^{1}(\Om)$.
\end{lemma}

\begin{proof}
	Let us prove~\eqref{22}. Let $u\in H^{1}(\Om)$ be such that $u=\bar{w}$ on~$\partial\Om$. By the interpolation error estimates in, e.g., \cite[Theroem~3.4.2]{Quarteroni1994}, for every $\varphi\in C^{\infty}_{c}(\Om)$ there exists a sequence $\varphi_{h}\in\F_{h}$ such that $\varphi_{h}=0$ on~$\partial\Om$ and $\varphi_{h}\to \varphi$ in $H^{1}(\Om)$ as $h\to \infty$. Let us consider as a competitor in~\eqref{20} the function $\psi_{h}\coloneq \varphi_{h}+w_{h}$. For such a $\psi_h$ we have
	\begin{equation}\label{24}
		\int_{\Om} \bigl( \p_{h}(v_{h}^{2}) + \eta \bigr) \abs{\nabla u_{h}}^{2} \,\di x \leq \int_{\Om} \bigl( \p_{h}(v_{h}^{2}) + \eta \bigr) \abs{\nabla \psi_{h}}^{2} \,\di x \,.
	\end{equation}
	
	It is clear that $\psi_{h} \to \psi \coloneq \varphi + \bar{w}$ in $H^{1}(\Om)$ and, by Lemma~\ref{lemma0}, $\p_{h}(v_{h}^{2}) \to \bar{v}^{2}$ strongly in~$L^{p}(\Om)$ for every $p\in[1, +\infty)$. Therefore, applying~\cite[Theorem 7.5]{Fonseca2007} and passing to the limit as $h \to +\infty$ in~\eqref{24}, we get
	\begin{align}
		\int_{\Om} (\bar{v}^{2} + \eta) \abs{\nabla \bar{u}}^{2} \,\di x &\leq \liminf_{h\to \infty} \int_{\Om} \bigl( \p_{h}(v_{h}^{2}) + \eta \bigr) \abs{\nabla u_{h}}^{2}\,\di x \nonumber\\
		&\leq \limsup_{h\to \infty} \int_{\Om} \bigl( \p_{h}(v_{h}^{2}) + \eta \bigr) \abs{\nabla \psi_{h}}^{2} \,\di x \nonumber \\
		&\leq \int_{\Om} \bigl( \bar{v}^{2} + \eta \bigr) \bigabs{\nabla \psi}^{2} \,\di x \,. \label{25}
	\end{align}
	By density, we have that the chain of inequalities~\eqref{25} holds for every $\psi \in \A(\bar w)$. Moreover, it is easy to see that~\eqref{25} is equivalent to~\eqref{22}.
	
	Specifying~\eqref{25} for~$\psi=\bar{u}$, we get that
	\begin{equation*}
		\lim_{h\to \infty} \int_{\Om} \bigl( \p_{h}(v_{h}^{2}) + \eta \bigr) \abs{\nabla u_{h}}^{2} \,\di x = \int_{\Om} (\bar{v}^{2} + \eta) \abs{\nabla \bar{u}}^{2} \,\di x\,,
	\end{equation*}
	which implies the strong convergence of~$u_{h}$ to~$\bar{u}$ in~$H^{1}(\Om)$.
	
	We now prove~\eqref{23}. Let us first consider a competitor $v\in H^{1}(\Om)\cap L^{\infty}(\Om)$, $v\leq\bar{v}$ in~$\Om$. Let $\varphi^{k}\in C^{\infty}(\overline{\Om})$ be such that $\varphi^{k}\leq0$ in~$\Om$ and $\varphi^{k}\to v-\bar{v}$ in~$H^{1}(\Om)$ as $k\to+\infty$. Let us set $v^{k}_{h}\coloneq \p_{h}(\varphi^{k})+v_{h}$. Then, $v^{k}_{h}\leq v_{h}$ for every $h,k\in\mathbb{N}$, $v_{h}^{k}\in\F_{h}$, and $v^{k}_{h}\rightharpoonup \bar{v}+\varphi^{k}$ weakly in~$H^{1}(\Om)$ as $h\to+\infty$. By the quadratic structure of~$\J_{h}$, by~\eqref{21}, and by the definition of~$v^{k}_{h}$,
	\begin{align}
		& \frac{1}{2} \int_{\Om} \bigl( \p_{h}(v_{h}^{2}) + \eta \bigr) \abs{\nabla u_{h}}^{2} \,\di x + \frac{1}{2} \int_{\Om} \p_{h}\bigl( (1-v_{h})^{2} \bigr) \,\di x \nonumber \\
		\leq{}& \frac{1}{2} \int_{\Om} \bigl( \p_{h}((v^{k}_{h})^{2}) + \eta \bigr) \abs{\nabla u_{h}}^{2} \,\di x + \frac{1}{2} \int_{\Om} \p_{h}\bigl( (1-v^{k}_{h})^{2} \bigr) \,\di x \nonumber \\
		&+ \frac{1}{2} \int_{\Om} \bigabs{\nabla \p_{h}(\varphi^{k})}^{2} \,\di x + \int_{\Om} \nabla \p_{h}(\varphi^{k}) \cdot \nabla v_{h} \,\di x\,. \label{26}
	\end{align}
	Since $u_{h}\to\bar{u}$ and $\p_{h}(\varphi^{k})\to\varphi^{k}$ strongly in~$H^{1}(\Om)$, $\p_{h}(v_{h}^{2})\to \bar{v}^{2}$ and $\p_{h}((v^{k}_{h})^{2})\to (\varphi^{k}+\bar{v})^{2}$ strongly in~$L^{p}(\Om)$ for every $p\in[1,+\infty)$ (see Lemma~\ref{lemma0}), and $v_{h}\rightharpoonup \bar{v}$ weakly in~$H^{1}(\Om)$, passing to the limit as $h\to+\infty$ in~\eqref{26} we deduce that
	\begin{align}
		&\frac{1}{2} \int_{\Om} (\bar{v}^{2} + \eta) \abs{\nabla \bar{u}}^{2} \,\di x + \frac{1}{2} \int_{\Om} (1-\bar{v})^{2} \,\di x \nonumber \\
		\leq{}& \frac{1}{2}  \int_{\Om} \bigl( (\varphi^{k} + \bar{v})^{2} + \eta \bigr) \abs{\nabla \bar{u}}^{2} \,\di x + \frac{1}{2}\int_{\Om} \bigl( 1 - (\varphi^{k} + \bar{v}) \bigr)^{2}\,\di x \nonumber \\
		&+ \frac{1}{2} \int_{\Om} \abs{\nabla \varphi^{k}}^{2} \,\di x + \int_{\Om} \nabla \varphi^{k} \cdot \nabla \bar{v} \,\di x\,. \label{27}
	\end{align}
	If we let $k\to+\infty$ in~\eqref{27}, recalling that $\varphi^{k}\to v-\bar{v}$ in~$H^{1}(\Om)$, we obtain
	\begin{align}
		& \frac{1}{2} \int_{\Om} (\bar{v}^{2} + \eta) \abs{\nabla \bar{u}}^{2} \,\di x + \frac{1}{2} \int_{\Om} (1 - \bar{v})^{2} \,\di x \nonumber \\
		\leq{}& \frac{1}{2} \int_{\Om} (v^{2} + \eta) \abs{\nabla \bar{u}}^{2} \,\di x + \frac{1}{2} \int_{\Om} (1-v)^{2} \,\di x \nonumber \\
		& + \frac{1}{2} \int_{\Om} \bigabs{\nabla (v-\bar{v})}^{2} \,\di x + \int_{\Om} \nabla (v-\bar{v}) \cdot \nabla \bar{v} \,\di x\,. \label{28}
	\end{align}
	Rearranging the last two terms in the right-hand side of~\eqref{28}, we get the stability condition~\eqref{23} for $v\in H^{1}(\Om)\cap L^{\infty}(\Om)$ with $v\leq\bar{v}$. By a truncation argument, we get the same conclusion for $v\in H^{1}(\Om)$ with $v\leq\bar{v}$.
\end{proof}

We are now ready to prove Theorem~\ref{mainthm}.

\begin{proof}[Proof of Theorem~\ref{mainthm}]
	As already mentioned in Remark~\ref{rmk2}, in order to prove the existence of a quasi-static evolution in the sense of Definition~\ref{quasi-staticcontinuous} we show that any sequence of finite-dimensional quasi-static evolutions $(u_{h},v_{h})\colon[0,T]\to\F_{h}$ converges, up to a subsequence, to a quasi-static evolution as the mesh parameter~$h$ tends to~$+\infty$.
	
	For every $h\in\mathbb{N}{\setminus}\{0\}$, we need first to find the right sequence of finite-dimensional quasi-static evolutions~$(u_{h},v_{h})$ starting from a suitable initial datum $u_{0,h}, v_{0,h}\in\F_{h}$ and with a suitable boundary Dirichlet condition $w_{h}\colon[0,T]\to\F_{h}$.
	
	As mentioned in Section~\ref{Discrete}, there exists a sequence $w_{h}\in W^{1,2}([0,T];H^{1}(\Om))$ such that $w_{h}\in W^{1,2}([0,T];\F_{h})$ and $w_{h}\to w$ in $W^{1,2}([0,T];H^{1}(\Om))$ as $h\to+\infty$ (see~\cite{Quarteroni1994}). In particular, the last convergence implies that $w_{h}(t)\to w(t)$ in $H^{1}(\Om)$ for every $t\in[0,T]$ and $\dot{w}_{h}(t)\to\dot{w}(t)$ in $H^{1}(\Om)$ for a.e. $t\in[0,T]$. Again by~\cite{Quarteroni1994}, we can also find two sequences~$u_{0,h}, v_{0,h}\in\F_{h}$ such that $u_{0,h}\in\A_{h}(w_h(0))$, $0\leq v_{0,h}\leq1$ in~$\Om$, and, as $h\to+\infty$, $u_{0,h}\to u_{0}$ and $v_{0,h}\to v_{0}$ in~$H^{1}(\Om)$.
	
	By Theorem~\ref{thm1}, for every~$h\in\mathbb{N}{\setminus}\{0\}$ there exists a finite-dimensional quasi-static evolution~$(u_{h},v_{h})\colon[0,T]\to\F_{h}{\,\times\,}\F_{h}$ with $u_{h}(0)=u_{0,h}$, $v_{h}(0)=v_{0,h}$, and $u_{h}(t) \in \A(w_{h}(t))$ for every $t\in[0,T]$.
	
	In view of~\eqref{stab1}--\eqref{energyineq} and of the construction of~$u_{0,h},v_{0,h}$, and~$w_{h}$, we have that
	\begin{equation}\label{18}
		\sup_{\substack{h>0\\ t\in[0,T]}} \norm{u_{h}(t)}_{H^{1}} < +\infty \qquad \text{and} \qquad \sup_{\substack{h>0\\ t\in[0,T]}} \norm{v_{h}(t)}_{H^1} <+\infty\,.
	\end{equation}
	As in the proof of Theorem~\ref{thm1}, since the sequence $v_{h}\colon[0,T]\to H^{1}(\Om)$ is such that~\eqref{18} holds,~$t\mapsto v_{h}(t)$ is non-increasing, and, for every $t\in[0,T]$,~$v_{h}(t)$ takes values in~$[0,1]$, applying a generalized version of Helly's Selection Theorem (see~\cite[Theorem~2.3]{Fleischer2001}, we find a non-increasing function $v\colon[0,T]\to H^{1}(\Om)$ such that, along a suitable subsequence $h_{k}\to+\infty$, for every $t\in[0,T]$ $v_{h_{k}}(t)$ converges to~$v(t)$ weakly in~$H^{1}(\Om)$ and strongly in~$L^{p}(\Om)$ for every $p\in[1,+\infty)$. In particular, $v(0)=v_{0}$ and $0\leq v(t)\leq1$ in~$\Om$ for every~$t\in[0,T]$, hence condition~(i) of Definition~\ref{quasi-staticcontinuous} is satisfied.
	
	In view of~\eqref{18}, for every $t\in[0,T]$ we have that, up to a subsequence (possibly dependent on~$t$), $u_{h_{k}}(t)\rightharpoonup u(t)$ weakly in~$H^{1}(\Om)$ for some $u(t)\in \A(w(t))$. By Lemma~\ref{lemma2}, we deduce that the pair $(u(t),v(t))$ satisfies the stability conditions~\eqref{contstabu} and~\eqref{contstabv}, and $u_{h_{k}}(t)\to u(t)$ strongly in $H^{1}(\Om)$. Moreover, by Lemma~\ref{lemma2} and by uniqueness of solution of the minimum problem
	\begin{equation*}
		\min\,\{\J(u,v(t)):\,u\in H^{1}(\Om),\,u\in\A(w(t))\}\,,
	\end{equation*}
	we have that the whole sequence~$u_{h_{k}}(t)$ converges to~$u(t)$ strongly in~$H^{1}(\Om)$ for every $t\in[0,T]$.
	
	Finally, in order to prove the energy inequality~\eqref{energyineq2}, we need to pass to the limit in the finite-dimensional energy inequality~\eqref{energyineq} as $h_{k}\to+\infty$. Since $u_{0,h}\to u_{0}$ and $v_{0,h}\to v_{0}$ and Lemma~\ref{lemma0} holds, we have that $\J_{h_{k}}(u_{0,h_{k}},v_{0,h_{k}})\to\J(u_{0},v_{0})$. Again by Lemma~\ref{lemma0}, $\p_{h_{k}}(v_{h_{k}}^{2}(t))\to v(t)$ in~$L^{p}(\Om)$ for every $p\in[1,+\infty)$ and every $t\in[0,T]$. Finally, by construction, $w_{h}\to w$ in $W^{1,2}([0,T]; H^{1}(\Om))$. Hence, passing to the limit in~\eqref{energyineq} as $h_{k}\to+\infty$ and applying the dominated convergence theorem, we get~\eqref{energyineq2}, and this concludes the proof of the theorem.
\end{proof}

\section{Numerical Experiments}\label{Numerics}
In this section we illustrate numerically the previous findings using the example of the phase-field model of brittle fracture. For these experiments we use the original Ambrosio-Tortorelli functional~$\J_{\varepsilon}$ defined in~\eqref{AmbTort}, whose discretized version is still denoted by~$\J_h$. We recall it here for the reader's convenience: for all $u,v \in \F_h$
\begin{equation} \label{discFunc}
	\J_h (u,v) \coloneq \frac{1}{2}\int_{\Omega} \bigl( \p_h(v^2) + \eta_\varepsilon \bigr) \abs{\nabla u}^2 \, \dd x + \! \int_{\Omega} \Big(\kappa \varepsilon\abs{\nabla v}^2 + \frac{\kappa}{4\varepsilon}\p_h \bigl( (1-v)^2 \bigr)\Big) \, \dd x \,.
\end{equation}
The following implementations are made with \texttt{Freefem++} \nocite{Hecht2012} and are performed on a MacBook Pro, 2.6 GHz Intel Core i7, 8 GB 1600 MHz DDR3. 

For all the examples in this section we fix the basic domain $\Omega \coloneq [0,1] {\,\times\,} [0,1]$ and we choose the time step $\tau \coloneq 0.01$. Hence, with an initial time $t_0 \coloneq 0$ we set $t_i \coloneq i\tau$. By~$v_i$ and~$u_i$ we denote the phase and the displacement field at time~$t_i$, respectively. We present three simulations for different boundary data~$w$ and initial cracks~$\Gamma$. The initialization of the crack is performed by removing the set~$\Gamma$ from our domain~$\Omega$. Hence, the calculations take place on the set $\widetilde \Omega \coloneq \Omega \setminus \Gamma$. Moreover, in the last example we influence the crack propagation by putting an additional hole into the domain. 

The basic algorithm we want to implement is the one given at the beginning of the proof of Theorem \ref{thm1}. Given $i\in\mathbb{N}\setminus\{0\}$ and the phase-field $v_{i-1}$ at time $t_{i-1}$, we set $v_{i,0}\coloneq v_{i-1}$,and we define inductively on $j\in\mathbb{N}$
\begin{align}
	u_{i,j} &\coloneq \argmin\, \{\J_{h}(u,v_{i,j-1}) : u\in\A_{h}(w_{h}(t_{i}))\} \label{minualg}\,,\\
	v_{i,j} &\coloneq \argmin\, \{\J_{h}(u_{i,j},v) : v\in\F_{h} ,\, v\leq v_{i,j-1}\} \label{minvalg} \,.
\end{align}
Due to the convergence of both sequences as $j\to+\infty$, we stop the loop when the phase-field does not show any significant changes. More precisely, we fix a threshold $0< \TOL_v \ll 1$, and we perform the alternate minimization~\eqref{minualg}-\eqref{minvalg} until we find $\bar{\jmath}\in\mathbb{N}$ such that $\norm{v_{i,\bar{\jmath}}-v_{i,\bar{\jmath}-1}}_{\infty}\leq \TOL_v$ and set $(u_i,v_i) \coloneq (u_{i,\bar{\jmath}},v_{i,\bar{\jmath}})$.

One of the main difficulties in the scheme described above is the numerical treatment of the irreversibility condition $v\leq v_{i,j-1}$ in~\eqref{minvalg}. Unfortunately, there is no way to implement this inequality constraint in a direct way like standard boundary conditions. In many papers (see, for instance,~\cite{Artina2015,Bourdin2000,Bourdin2007a,Bourdin2007,Burke2010}) this difficulty has been overcome by replacing the irreversibility constraint with a Dirichlet boundary condition, forcing~$v$ to be zero where $v_{i,j-1}$ is below a certain threshold. The method turned out to be numerically very efficient; however, up to our knowledge there is a lack of a rigorous theoretical proof of convergence of the scheme in the framework of quasi-static evolution problems.

In order to be as close as possible to the theoretical result presented in this paper, in our numerical simulations we do not exploit the algorithm proposed in~\cite{Bourdin2000}, but we rather enforce the inequality constraint in a weak sense by adding to the original discretized functional~$J_{h}(u,v)$ in ~\eqref{discFunc} a term that penalizes phase-fields that are larger than~$v_{i,j-1}$. Hence, instead of minimizing~$\J_h (u,v)$ under the constraint $v \leq v_{i,j-1}$ in~\eqref{minvalg}, we target an unconstrained minimization of the following functional
\begin{equation*}
	\widetilde \J_h (u,v) := \J_h (u,v) + \frac{\zeta}{2} \int_\Omega \p_h\bigl( [v-v_{i,j-1}]_+^2 \bigr) \,\dd x \,.
\end{equation*}
Here, $\zeta \gg 1$ is a penalization constant and $[\cdot]_+$ denotes the positive part function. It is straightforward to see that $\widetilde \J_h (u,\cdot)$  approximates $\J_h(u,\cdot)$ in the sense of $\Gamma$\nobreakdash-converges for $\zeta \to \infty$, telling us that the minima of $\widetilde \J_h (u,\cdot)$ converge to minima of $\J_h (u, \cdot)$ under the condition $v \leq v_{i, j-1}$ as $\zeta \to \infty$.

The functional~$\widetilde \J_h(u,\cdot)$ is still strictly convex and coercive. Therefore, it has a unique minimizer. 
In order to find the minimum point of~$\widetilde \J_h(u,\cdot)$, it is very convenient to solve the Euler-Lagrange equation
\begin{displaymath}
0 = \int_{\Om}  v \psi \bigabs{\nabla  u}^{2} \,\di x + \int_{\Om} \nabla  v \cdot \nabla \psi \,\di x - \int_{\Om} \bigl( 1 -  v \bigr) \psi \,\di x + \zeta \int_\Omega [v-v_{i-1}]_+  \psi \,\dd x
\end{displaymath}
for every $\psi \in \F_h$.

Due to the penalization term, this equation is non-linear, and hence it is not possible to solve it directly using the stiffness matrix. Therefore, instead of using a direct solver, we use the Newton method: given~$v_{i,j-1}$, we set~$v^0_{i,j}\coloneq v_{i,j-1}$ and for each~$\ell \in \N$
\begin{align*}
v^\ell_{i,j} \coloneq v^{\ell-1}_{i,j} - \bigl( \nabla_v^2 \widetilde \J_h (u_{i,j}, v_{i,j}^{\ell-1}) \bigr)^{-1} \cdot \nabla_v \widetilde{\J}_h (u_{i,j}, v_{i,j}^{\ell-1}) \,.
\end{align*}
The iteration is stopped as soon as the change $\norm{v^{\bar\ell}_{i,j} - v^{\bar{\ell}-1}_{i,j}}_\infty$ and the value $\norm{\nabla_v \widetilde \J_h (u_{i,j}, v^{\bar \ell}_{i,j})}_\infty$ are sufficiently small for some $\bar\ell\in\N$ and define $v_{i,j} \coloneq v^{\bar\ell}_{i,j}$.

On most parts of the domain the phase-field will be nearly constant. Only close to the crack it is expected to be very steep. To get an appropriate interpolation error the mesh has to be very fine in the neighborhood  of the crack, while it can be coarse elsewhere. Thus, we use an adaptive triangulation refining the mesh where it is necessary. Such approaches has been particularly investigated in~\cite{Artina2015,Burke2010}. For our purposes, we regularly adapt the mesh in the iteration procedure using the standard routine implemented in \texttt{Freefem++}, which uses the standard anisotropic interpolation error estimator.

In what follows, we present three different examples.The basic numerical scheme we use is described in Algorithm \ref{pseudocode} and all the appearing numerical parameters are fixed to the values in Table \ref{numParam}.
As an output, at each time step we visualize the number of iterations in the alternating minimization scheme (corresponding to the middle while-loop in Algorithm \ref{pseudocode}) and the number of required iterations in the Newton method (corresponding to the inner while-loop). Moreover, we present the crack length with respect to time, which we approximate by $\frac{1}{\varepsilon} \int_{\Omega} (1-v_i) \,\dd x$.
\begin{algorithm}[htbp]
	\begin{algorithmic}
		\For{$i=1$ to $k$}
		\State $v_{i} \leftarrow v_{i-1}$
		\Do
			\State $\mathit{cnt} \leftarrow 0$
			\Do
				\State $\mathit{cnt} \leftarrow \mathit{cnt} + 1$
				\State $v_\text{old} \leftarrow v_i$
				\State $u_i \leftarrow \argmin \bigl\{\J_h (u, v_\text{old}) \mid u \in \A_h(w(t_i)) \bigr\}$
				\Do
					\State $\hat v_\text{old} \leftarrow v_i$
					\State $\displaystyle v_i \leftarrow \hat v_\text{old} - \bigl( \nabla_v^2 \widetilde \J_h (u_i, \hat v_\text{old} ) \bigr)^{-1} \cdot \nabla_v \widetilde{\J}_h (u_i, \hat v_\text{old}) \,.$
				\doWhile{$\norm{v_i - \hat v_{\text{old}}}_\infty > \TOL_v$ AND $\norm{\nabla_v \widetilde \J_h (u_i, v_i) }_\infty >\TOL_v$}
			\doWhile{$\norm{v_i -  v_\text{old}}_\infty > \TOL_v$ AND $\mathit{cnt} < 10$}
			\State Perform the mesh adaption
			\State $\mathit{rel}_\text{mesh} \leftarrow$ ''relative change of nodes''
		\doWhile{$\mathit{rel}_\text{mesh} > \TOL_\text{mesh}$}
		\EndFor
	\end{algorithmic}
	\caption{}\label{pseudocode}
\end{algorithm}

\begin{table}[htbp]
	\centering
	\caption{Numerical Parameters}\label{numParam}
	\begin{tabular}{cccccc}
		\hline
		$\varepsilon$ & $\eta_\varepsilon$ & $\kappa$ &  $\zeta$ & $\TOL_v$ & $\TOL_\text{mesh}$ \\
		\hline 
		$2 \cdot 10^{-2}$ & $10^{-5}$ & $0.5$ & $10^6$ & $2\cdot10^{-3}$ & $10^{-2}$ \\
		\hline
	\end{tabular}
\end{table}

\subsection*{First Example}
Our first example starts from a fully symmetrical setting, where we impose an initial crack orthogonal to and in the middle of the left boundary of the domain $\Omega$ described by the set
\begin{equation*}
	\Gamma_1 := [0,0.125]\times \bigl( 0.5(1-10^{-3}), 0.5(1 + 10^{-3}) \bigr) \,.
\end{equation*}
We also use a symmetric boundary condition
\begin{equation*} 
	w_1(t) = \begin{cases}
		t & \text{on } \{0\} \times [0,0.5(1-10^{-3})] \,, \\
		-t & \text{on } \{0\} \times [0.5(1+10^{-3}),1] \,.
	\end{cases}
\end{equation*}
The different measured values are visualized in Figure \ref{firstExample}. The crack propagates until the domain is fully broken at time $0.85$. In Figure \ref{firstExamplePhaseField} one can see the corresponding phase-field  at different time steps.
\begin{figure}[htbp]
	\centering
	\footnotesize
	\input{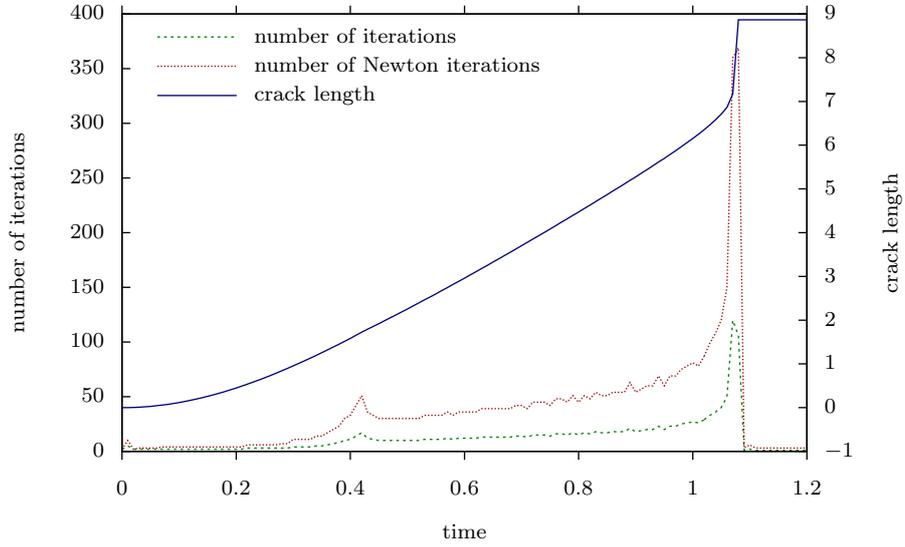}
	\normalsize
	\caption{\label{firstExample}Numerical output data of the first Example with initial crack $\Gamma_1$ and boundary condition $w_1$. The total calculation time until $t=1.08$ is approx. 2.0 hours.}
\end{figure}
\begin{figure}[htbp]
	\centering
	\begin{subfigure}{0.3\textwidth}
		\includegraphics[width=\textwidth]{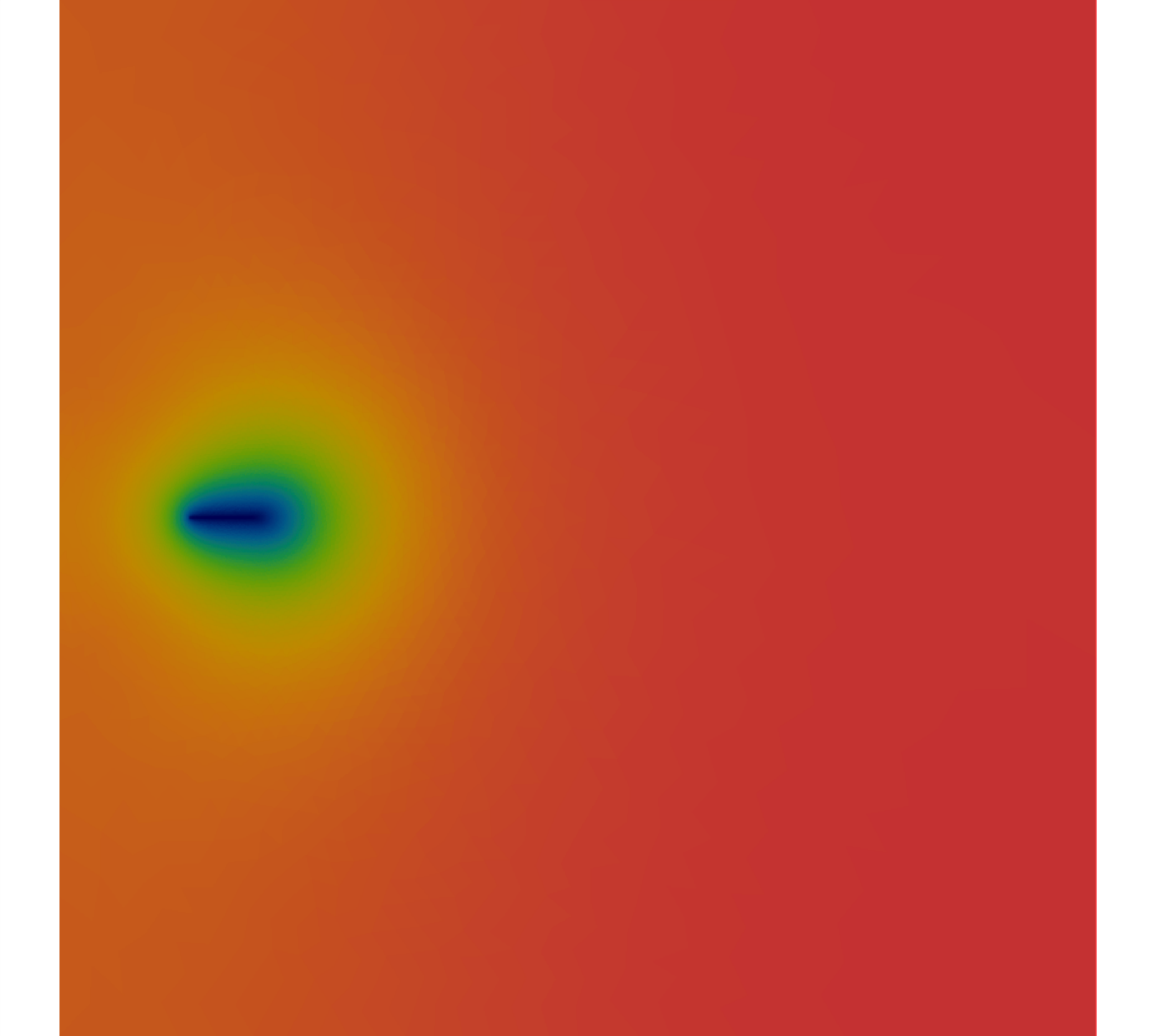}
		\caption{$t=0.6$}
	\end{subfigure}
	\begin{subfigure}{0.3\textwidth}
		\includegraphics[width=\textwidth]{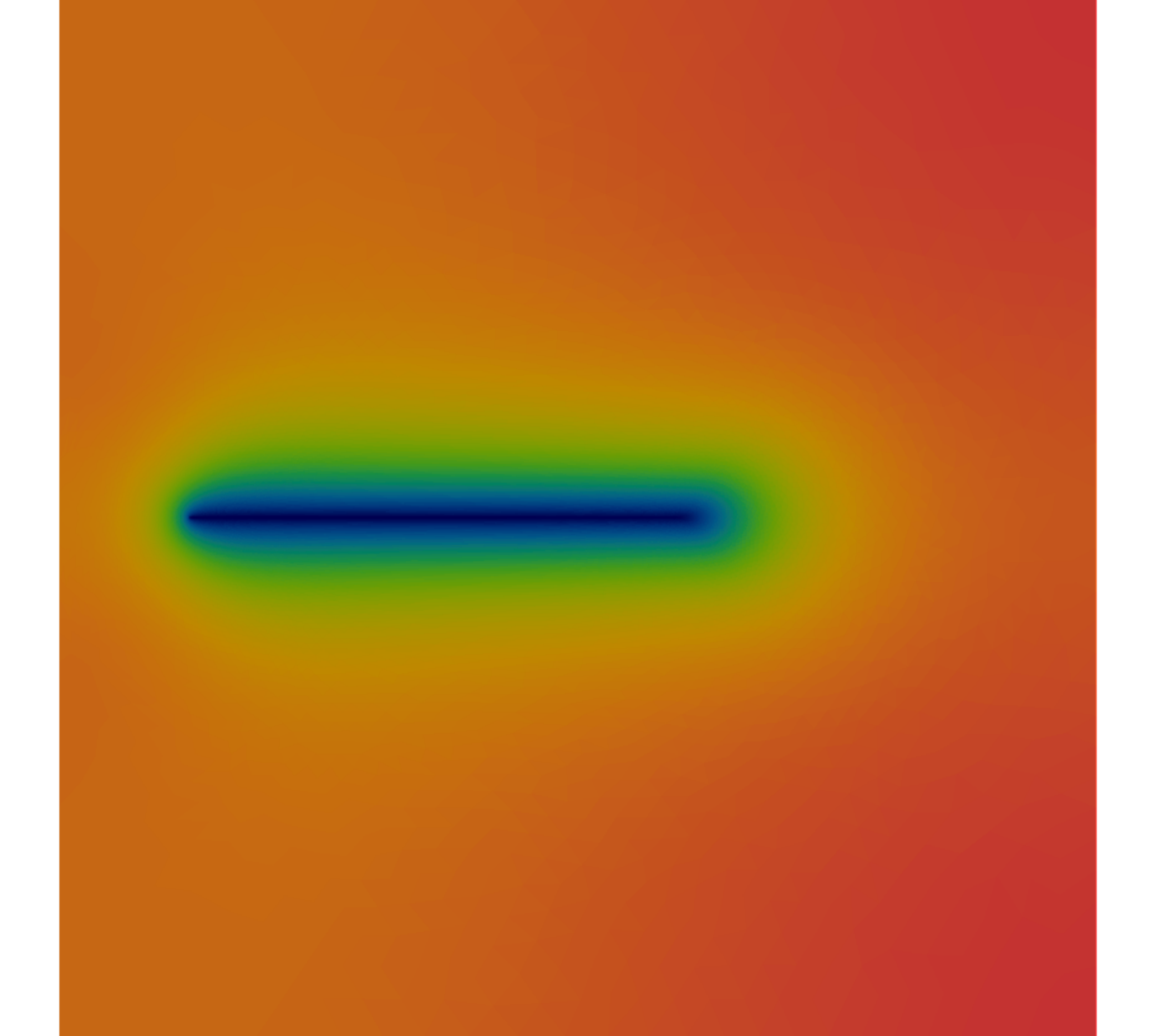}
		\caption{$t=1.07$}
	\end{subfigure}
	\begin{subfigure}{0.3\textwidth}
		\includegraphics[width=\textwidth]{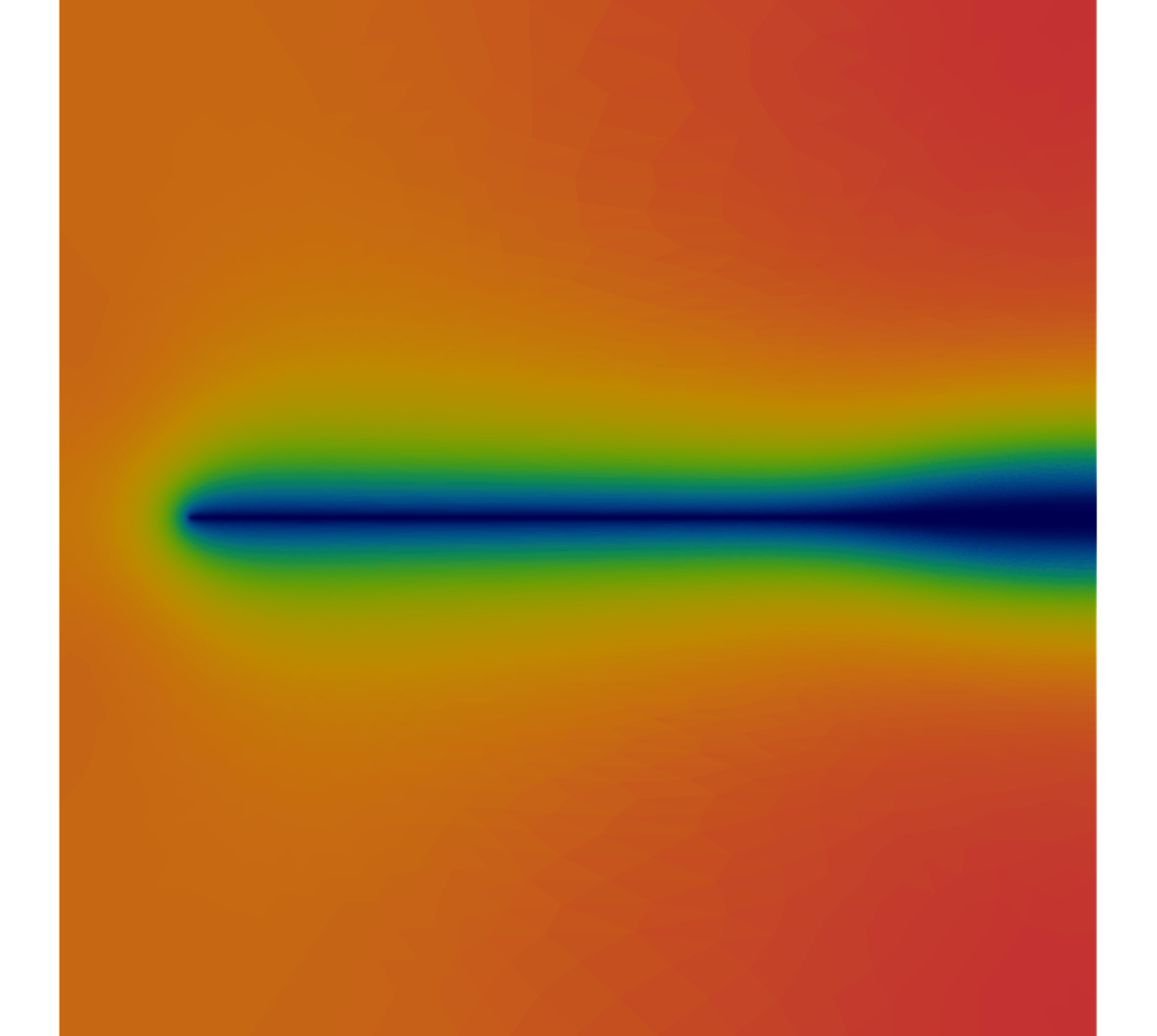}
		\caption{$t=1.08$}
	\end{subfigure}
	\caption{\label{firstExamplePhaseField}Phase field of first example at various times.}
\end{figure}

\subsection*{Second Example}
In our second example we break the symmetries by turning the initial crack into a set not orthogonal to the left boundary and shifted out of the middle, namely,
\begin{equation*}
	\Gamma_2 := \biggl\{ (x,y) \in \R^2 \mid x\in \biggl[0,\frac{1}{8}\biggr],   \frac{1}{2}x + \frac{2}{5} - \frac{10^{-3}}{2}< y <  \frac{1}{2}x + \frac{2}{5} + \frac{10^{-3}}{2} \biggr\} \,.
\end{equation*}
The boundary condition stays the same in the sense that it is $t$ below the slit and~$-t$ above it. Thus, we have
\begin{equation*} 
	w_2(t) = \begin{cases}
	t & \text{on } \{0\} \times [0,0.4-0.5\cdot 10^{-3}] \,,\\
	-t & \text{on } \{0\} \times [0.4+0.5 \cdot 10^{-3},1] \,.
	\end{cases}
\end{equation*}
With this settings, the specimen completely breaks after the 72nd time step, where the crack crosses the bottom border as it can be seen in Figure \ref{secondExamplePhaseField}. Figure \ref{secondExample} show the corresponding numerical measurements.
\begin{figure}[bhtp]
	\centering
	\footnotesize
	\input{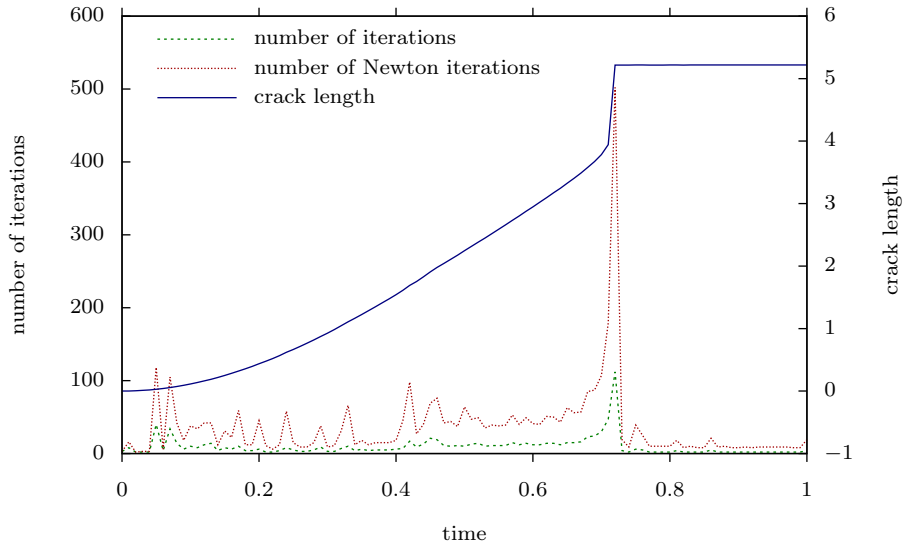}
	\normalsize
	\caption{\label{secondExample}Numerical output data of the second Example with initial crack $\Gamma_2$ and boundary condition $w_2$. The total calculation time until $t=0.72$ is approx. 1.7 hours.}
\end{figure}
\begin{figure}[htbp]
	\centering
	\begin{subfigure}{0.3\textwidth}
		\includegraphics[width=\textwidth]{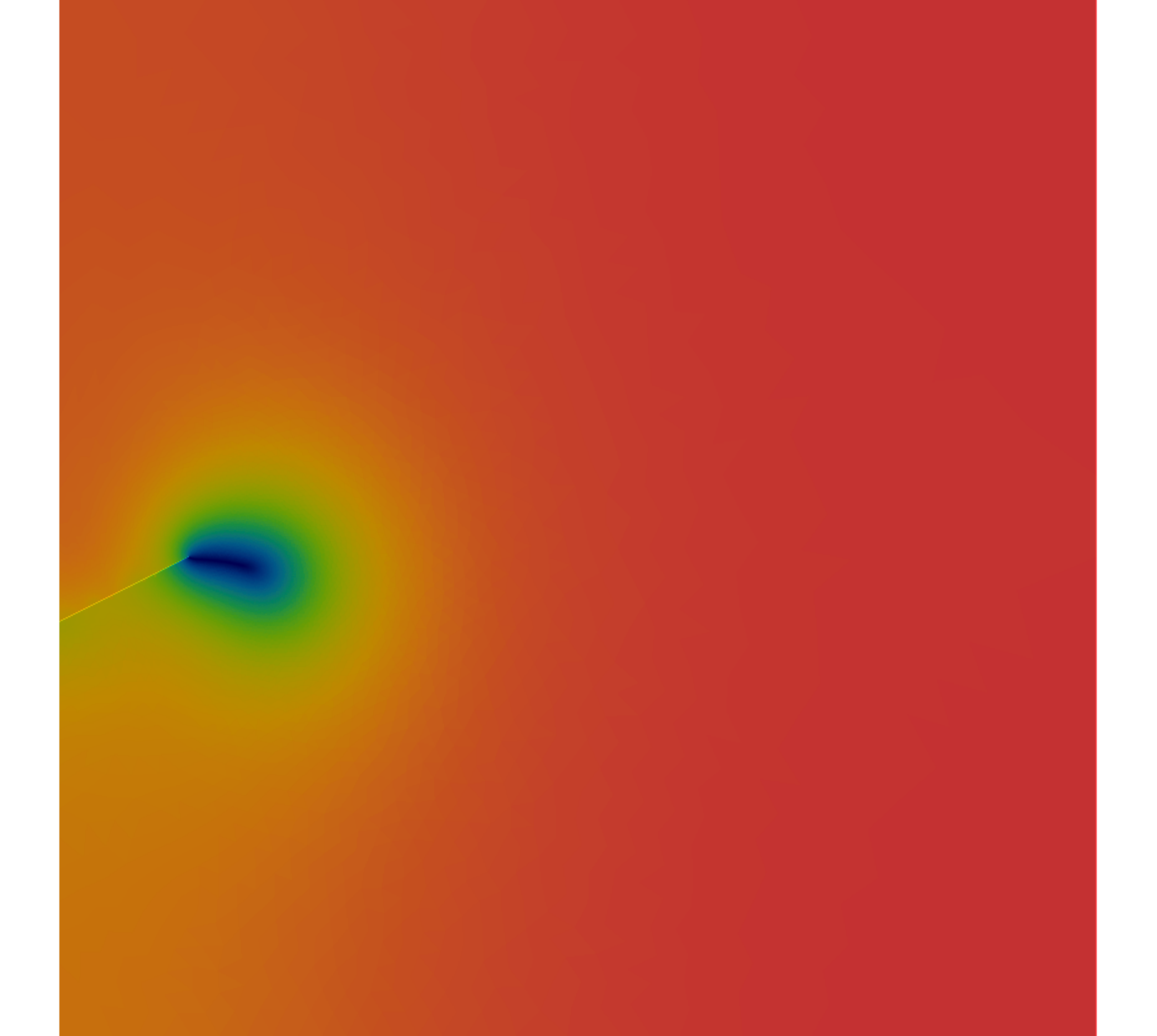}
		\caption{$t=0.6$}
	\end{subfigure}
	\begin{subfigure}{0.3\textwidth}
		\includegraphics[width=\textwidth]{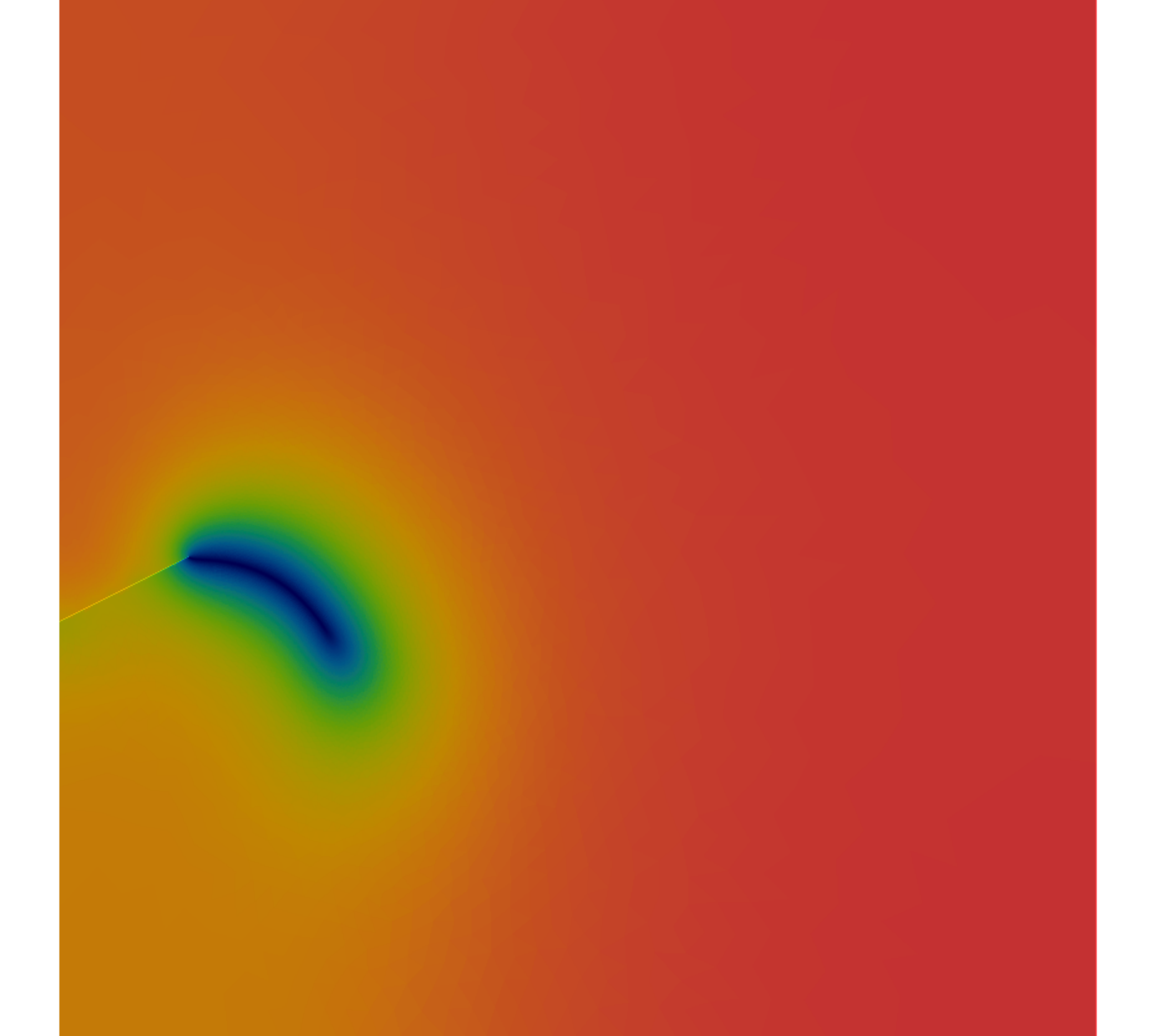}
		\caption{$t=0.71$}
	\end{subfigure}
	\begin{subfigure}{0.3\textwidth}
		\includegraphics[width=\textwidth]{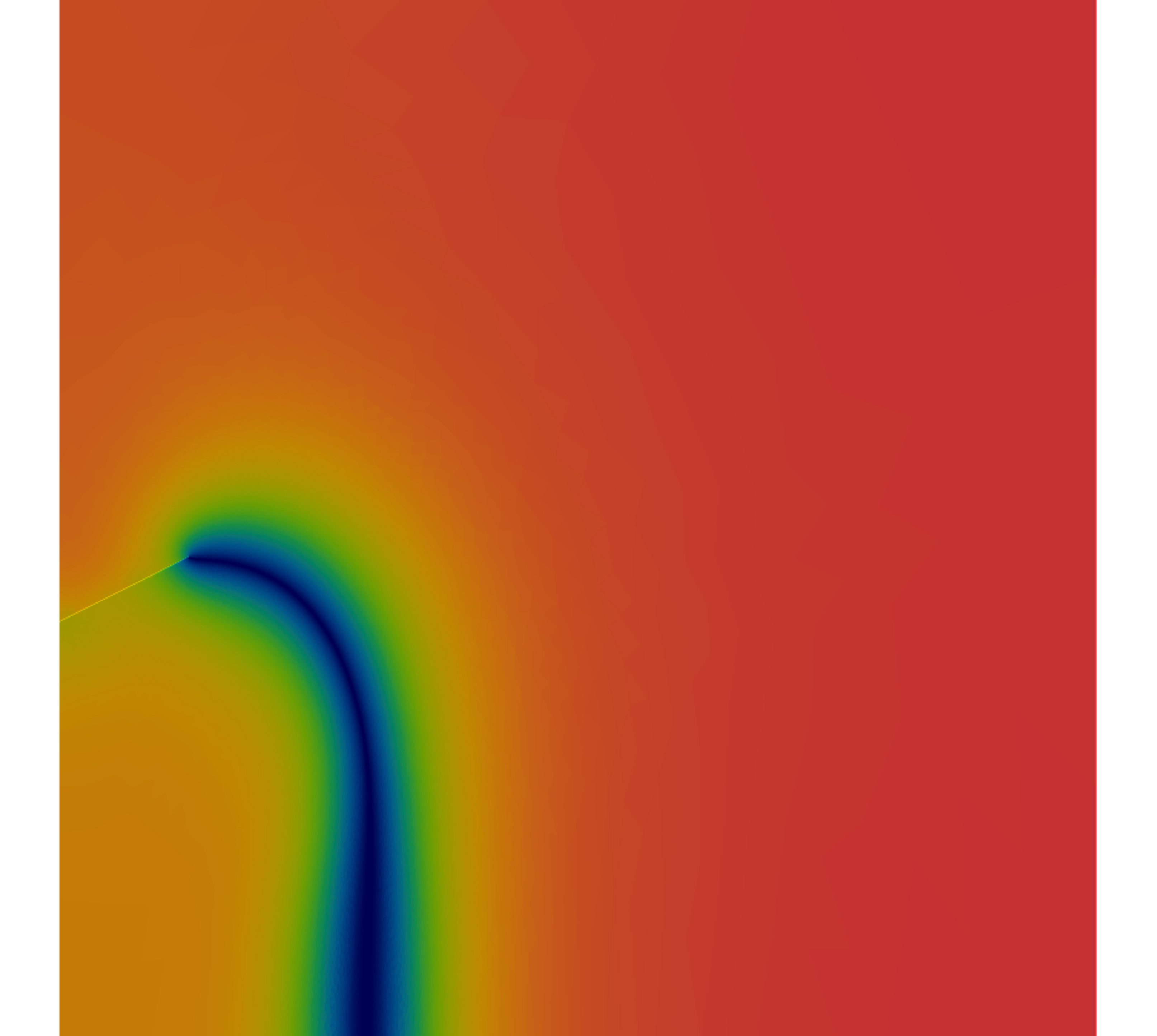}
		\caption{$t=0.72$}
	\end{subfigure}
	\caption{\label{secondExamplePhaseField}Phase field of second example at various times.}
\end{figure}

\subsection*{Third Example}
For the last example we again use $\Gamma_1$ as a pre-crack and $w_1$ has the driving boundary condition. Unlike the first example, we cut out a whole of the domain with center $(0.7,0.3)$ and radius $0.1$. This setting makes the crack deviate from the middle line into the hole, as it can be seen in Figure \ref{thirdExamplePhaseField}. The last part of the crack from the border of the hole to the border of the domain appears instantaneously at time $1.11$.
\begin{figure}[htbp]
	\centering
	\footnotesize
	\input{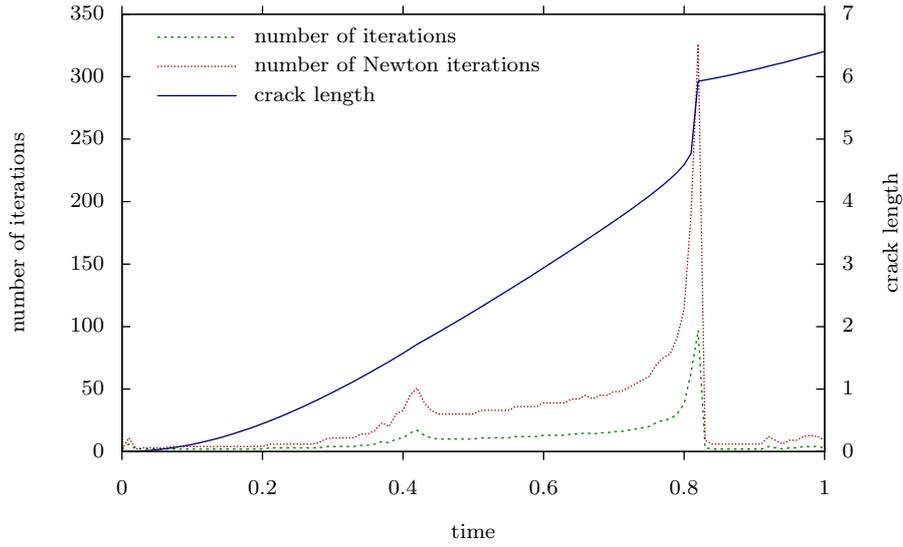}
	\normalsize
	\caption{\label{thirdExample}Numerical output data of the third Example with initial crack $\Gamma_1$ and boundary condition $w_1$ and the hole $B_{(0.7,0.3)}(0.1)$ cut out of the domain. The total calculation time until $t=1.11$ is approx. 1.0 hours.}
\end{figure}
\begin{figure}[htbp]
	\centering
	\begin{subfigure}{0.29\textwidth}
		\includegraphics[width=\textwidth]{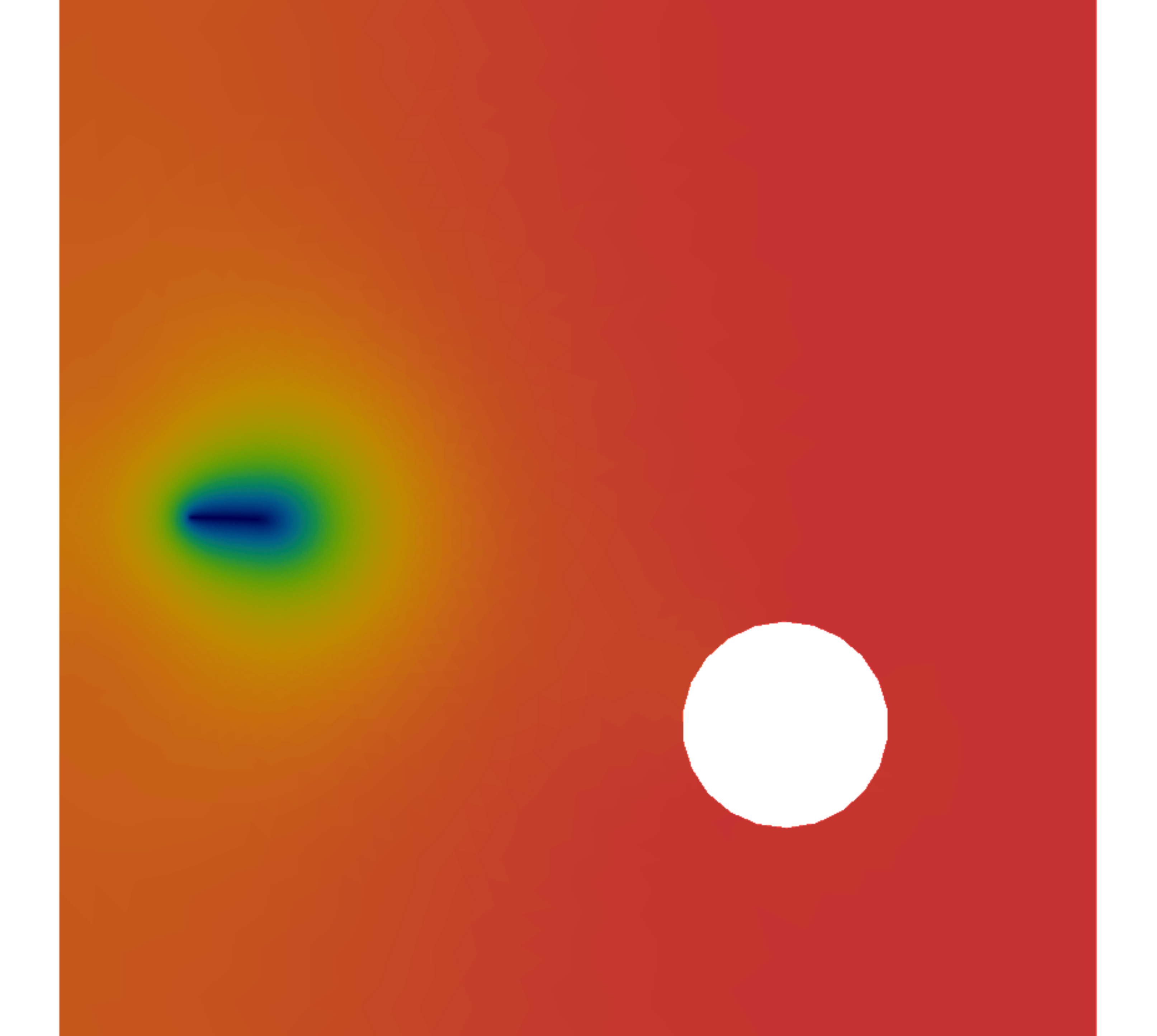}
		\caption{$t=0.6$}
	\end{subfigure}
	\begin{subfigure}{0.29\textwidth}
		\includegraphics[width=\textwidth]{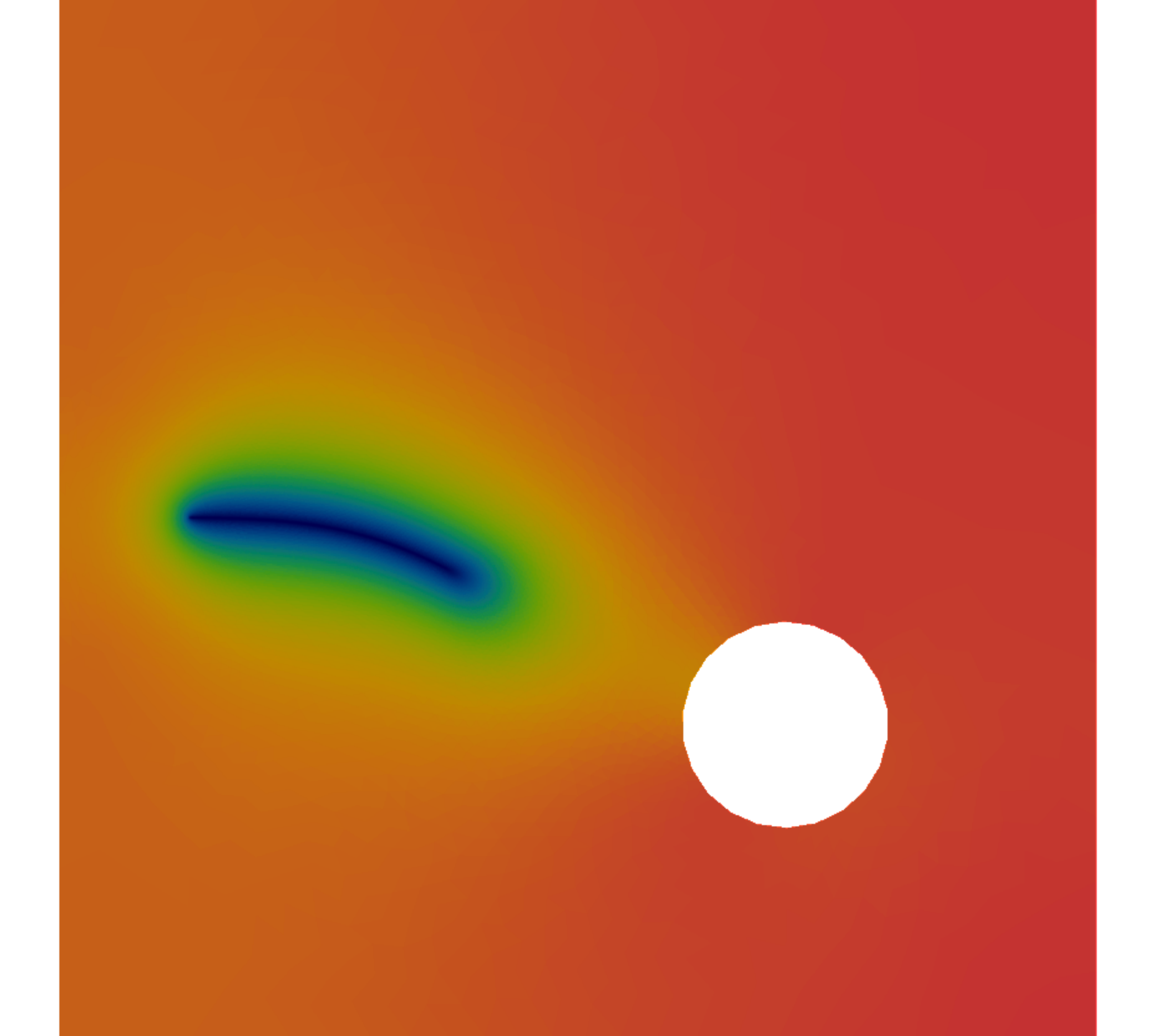}
		\caption{$t=0.81$}
	\end{subfigure}\\
	\begin{subfigure}{0.29\textwidth}
		\includegraphics[width=\textwidth]{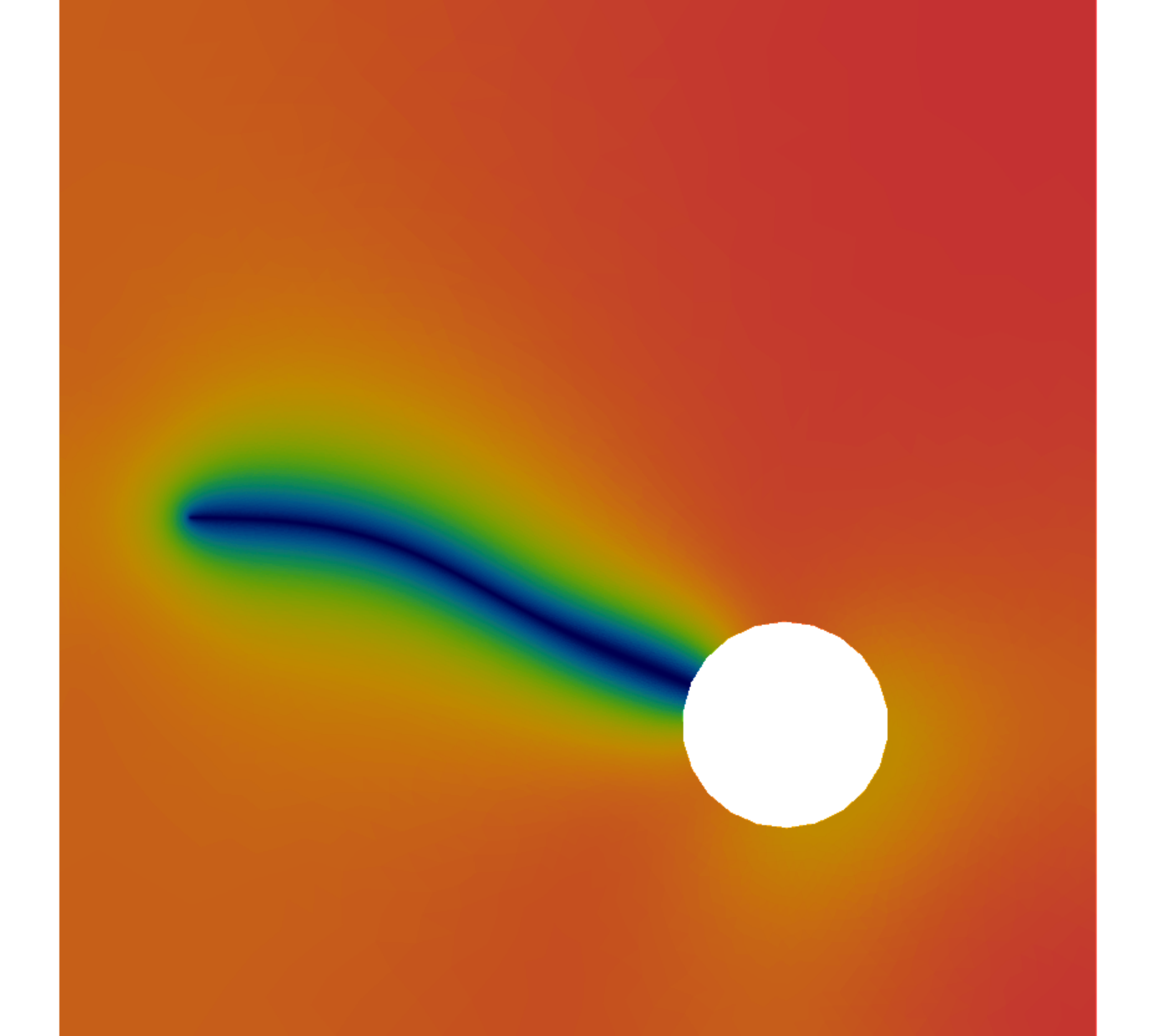}
		\caption{$t=0.82$}
	\end{subfigure}
	\begin{subfigure}{0.29\textwidth}
		\includegraphics[width=\textwidth]{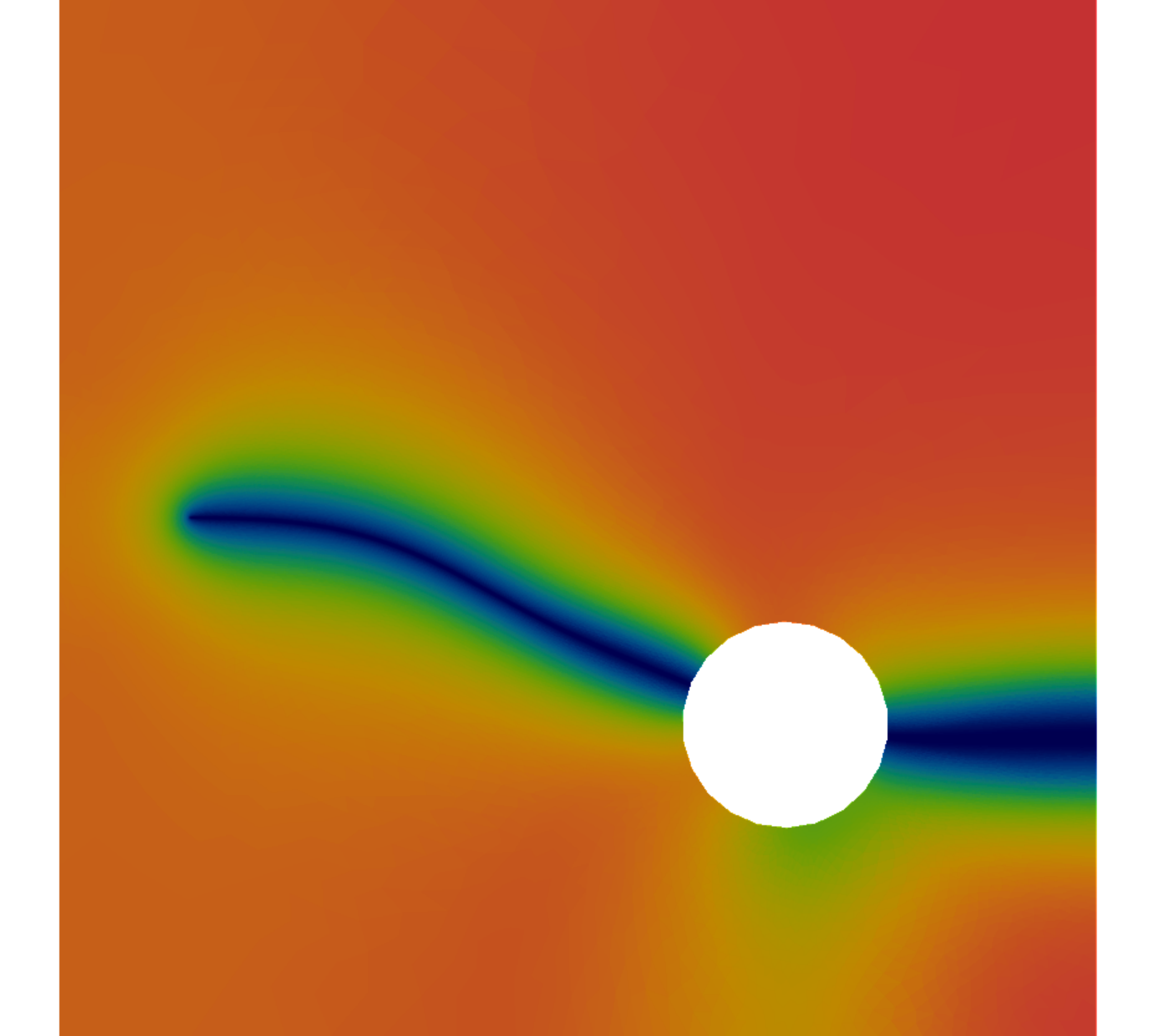}
		\caption{$t=1.11$}
	\end{subfigure}
	\caption{\label{thirdExamplePhaseField}Phase field of third example at various times.}
\end{figure}

\appendix
\section{}
~
\begin{proof}[Proof of Proposition \ref{Prop1}]\label{ProofProp1}
	It is enough to prove the statement for $j=1$. In order to show the existence of a minimizer for~\eqref{minv} we want to apply the direct method of the calculus of variations. Since $\J_{h}(u_{1},\cdot)$ is continuous with respect to the convergence in~$\F_{h}$, we only need to show that a minimizing sequence~$z_{k}\in\F_{h}$ for~\eqref{minv} admits a limit, at least up to a subsequence. Since
	\begin{equation*}
		\sup_{k\in\mathbb{N}} \,\norm{\nabla z_{k}}_{2}^{2} \leq \sup_{k\in\mathbb{N}} \, 2\J_{h}(u_{1},z_{k}) < +\infty \,,
	\end{equation*}
	by construction of the function space~$\F_{h}$~\eqref{discretespace} and of the triangulation~$\T_{h}$, we easily deduce that
	\begin{equation}\label{linftybound}
		\sup_{k} \,\norm{\nabla z_{k}}_{\infty} < + \infty \,,
	\end{equation}
	so that the functions~$z_{k}$ are uniformly Lipschitz in~$\Om$. Moreover, since by hypothesis~$v_{0}\geq0$, it is not restrictive to assume that there exists at least one vertex~$x_{l}$, $l\in\{1,\ldots,N_{h}\}$, such that $z_{k}(x_{l})\geq0$. Indeed, if~$z_{k}<0$ in $\Om$, it is readily seen that $\J_{h}(u_{1},0)\leq\J_{h}(u_{1},z_{k})$. Therefore, thanks to~\eqref{linftybound} and to the inequality $z_{k}\leq v_{0}$ in~$\Om$, we also deduce that the sequence~$z_{k}$ is uniformly bounded in~$W^{1,\infty}(\Om)$. Thus,~$z_{k}$ converges, up to a subsequence, to some~$v_{1}$ in~$\F_{h}$, and this concludes the proof of existence. The uniqueness of solution follows by the strict convexity of the functional~$\J_{h}(u_{1},\cdot)$.
	
	We now prove the second part of the statement, i.e., that $0\leq v_{1}\leq 1$. For the sake of contradiction, let us first assume that~$v_{1}\ngeq0$. Using the notation described in Section~\ref{Discrete}, let~$x_{l}\in\Delta_{h}$ be such that $v_{1}(x_{l})\leq v_{1}(x_{m})$ for every $m=1,\ldots,N_{h}$. In particular, we have~$v_{1}(x_{l})<0$. Let $\xi_{l}\geq0$ be the $l$-th element of the basis of~$\F_{h}$ defined by~\eqref{basis}. Being $v_{0}\geq0$ in $\Om$ and~$v_{1}(x_{l})<0$, for every~$\varepsilon\in\R$ with~$|\varepsilon|$ small enough we have that $v_{1}+\varepsilon\xi_{l}\leq v_{0}$ in~$\Om$, and, by the minimality of~$v_{1}$, $\J_{h}(u_{1},v_{1})\leq \J_{h}(u_{1},v_{1}+\varepsilon\xi_{l})$. By the quadratic structure of~$\J_{h}(u_{1},\cdot)$, from the previous inequality and the arbitrariness of $\varepsilon$ we deduce that
	\begin{equation}\label{discretecritical2}
		\int_{\Om} \p_{h} (v_{1}\xi_{l}) \abs{\nabla u_{1}}^{2} \,\di x + \int_{\Om} \nabla v_{1} \cdot \nabla \xi_{l} \di x - \int_{\Om} \p_{h} \bigl( (1-v_{1}) \xi_{l}\bigr) \,\di x= 0 \,.
	\end{equation}
	
	Since~$v_{1}(x_{l})<0$ and~\eqref{basis} holds, we have that
	\begin{equation}\label{contr0}
		\p_{h} (v_{1}\xi_{l}) \leq 0 \quad \text{and} \quad \p_{h} \bigl( (1 - v_{1}) \xi_{l} \bigr) \geq 0 \qquad \text{in~$\Om$}.
	\end{equation}
	Hence, from~\eqref{discretecritical2} and~\eqref{contr0} we get 
	\begin{equation}\label{contr1}
		\int_{\Om} \nabla v_{1} \cdot \nabla \xi_{l} \,\di x = -\int_{\Om} \p_{h} (v_{1}\xi_{l}) \abs{\nabla u_{1}}^{2} \,\di x + \int_{\Om}\p_{h} \bigl( (1-v_{1}) \xi_{l} \bigr) \,\di x \geq 0 \,.
	\end{equation}
	
	On the other hand, we can write~$v_{1}$ as a linear combination of the elements of the basis~$\{\xi_{m}\}_{m=1}^{N_{h}}$ of~$\F_{h}$, namely, $v_{1}=\sum_{m=1}^{N_{h}}v_{1}(x_{m})\xi_{m}$. Then, by direct computation,
	\begin{align}
		&\int_{\Om}  \nabla v_{1} \cdot \nabla \xi_{l} \,\di x = \sum_{m=1}^{N_{h}} v_{1}(x_{m}) \int_{\Om} \nabla \xi_{m} \cdot \nabla \xi_{l} \,\di x \nonumber \\
		={}& v_{1}(x_{l}) \sum_{m=1}^{N_{h}} \int_{\Om} \nabla \xi_{m} \cdot \nabla \xi_{l} \,\di x + \sum_{m=1}^{N_{h}} \bigl( v_{1}(x_{m}) - v_{1}(x_{l}) \bigr) \int_{\Om} \nabla \xi_{m} \cdot \nabla \xi_{l} \,\di x \nonumber \\
		={}& \sum_{m=1}^{N_{h}} \bigl( v_{1}(x_{m}) - v_{1}(x_{l}) \bigr) \int_{\Om} \nabla \xi_{m}  \cdot\ \nabla \xi_{l} \,\di x\leq 0\,, \label{contr2}
	\end{align}
	where, in the last equality, we have used~\eqref{stiffness}, the particular choice of the vertex~$x_{l}$, and the fact that
	\begin{equation*}
		\sum_{m=1}^{N_{h}}\int_{\Om}\nabla{\xi_{m}}{\,\cdot\,}\nabla{\xi_{l}}\,\di x=0\,.
	\end{equation*}
	Therefore, combining~\eqref{contr1} and~\eqref{contr2} we get a contradiction, and thus~$v_{1}\geq0$.
	
	In order to show that $v_{1}\leq 1$, we can argue again by contradiction, assuming that there exists a vertex~$x_{l}$, $l\in\{1,\ldots,N_{h}\}$ such that $v_{1}(x_{l})\geq v_{1}(x_{m})$ for every $m=1,\ldots,N_{h}$ and $v_{1}(x_{l})>1$. As before, being $\xi_l$ the $l$\nobreakdash-th element of the basis of $\F_h$, for $\varepsilon \in \R$ with $\abs{\varepsilon}$ small enough there holds $v_1 +\varepsilon \xi_l \geq v_0$, such that we can repeat~\eqref{discretecritical2}\nobreakdash--\eqref{contr2} with opposite inequality signs, concluding that $0\leq v_{1}\leq 1$ in~$\Om$. This concludes the proof of the proposition.
\end{proof}

\begin{proof}[Proof of Lemma \ref{lemma0}]\label{ProofLemma0}
	Let~$M>0$,~$g\in C^{2}(\R)$, the mesh parameter
	$h\in \N\setminus\{0\}$, and the triangulation~$\T_{h}$ be fixed, and let us consider a function~$v\in\F_{h}$ such that~$\|v\|_{\infty}\leq M$.
	
	Given an element~$K$ of the triangulation~$\T_{h}$, we first prove~\eqref{discreteestimate} on~$K$. Without loss of generality, we may assume that the origin is a vertex of~$K$ and
	\begin{equation*}
		K =\bigl\{ x = (x_{1},x_{2}) \in \Om : x_{1} \in [0,h^{-1}] ,\, x_{2} \in [0,x_{1}] \bigr\}\,.
	\end{equation*}
	The general case follows by an affine transformation.
	
	Recalling that if a function~$v\in\F_{h}$, then it is affine on every element of~$\T_{h}$, for every $x\in K$ we have
	\begin{equation*}
		v(x) = v(0) + \nabla v\rvert_K \cdot x \,.
	\end{equation*}
	Hence, by Taylor expansion, there exists some $\xi \in K$ such that
	\begin{equation}\label{a1}
		g ( v(x) )  = g( v(0) ) + g' ( v(0) ) \bigl( v(x) - v(0) \bigr) + \frac{1}{2} g''( v(\xi) ) \bigl( v(x) - v(0)\bigr)^2\,.
		\end{equation}
In particular, the last term in~\eqref{a1} can be simply estimated by
	\begin{equation}\label{a2}
		|g''( v(\xi) )|\bigl( v(x) - v(0)\bigr)^2\leq |g''( v(\xi) )| \bigabs{\nabla v \rvert_K }^2 \abs{x}^2 \,.
	\end{equation}
	 Recalling the definition of~$\p_{h}$~\eqref{ph} and using the expansion~\eqref{a1} we can write
	\begin{align}\label{a3}
		\p_{h}( g \circ v ) (x)  ={}& g( v(0) ) + h\bigl(g( v(h^{-1},0) ) - g( v(0) )\bigr) x_{1} \nonumber \\
		& + h\bigl(g( v(0,h^{-1}) ) - g( v(0) )\bigr) x_{2} \nonumber\\
		={}& g( v(0)) + h g'( v(0) ) \biggl((v(h^{-1},0) - v(0)) x_{1} + (v(0,h^{-1}) - v(0)) x_{2} \biggr)  \nonumber\\
		&+ \frac{h}{2}g''(v(\xi_{1})) \bigabs{v(h^{-1},0) - v(0)}^2 x_{1} \nonumber\\
		& +\frac{h}{2}g''(v(\xi_{2})) \bigabs{v(0,h^{-1}) - v(0)}^2  x_{2} 
	\end{align}
	for some~$\xi_{1},\xi_{2}\in K$. Hence, in view of~\eqref{a1}-\eqref{a3}, for $x\in K$ there holds
	\begin{align*}
		\bigabs{ g \bigl( v (x) \bigr) - \p_{h} ( g \circ v  ) (x)} \leq C h^{-2} \bigabs{\nabla v \rvert_K }^2
	\end{align*}
	for some positive constant~$C=C(g,M)$ depending only on~$g\in C^{2}(\R)$ and on~$M$. Integrating the previous inequality, we end up with~\eqref{discreteestimate}, and the proof is concluded. 
\end{proof}

\section*{Acknowledgements}
 This material is based on work supported by the European Research Council under Grant No. 501086
``High-Dimensional Sparse Optimal Control''. S.A. is member of the Gruppo Nazionale per l'Analisi Matematica, la Probabilit\`a e le loro Applicazioni (GNAMPA) of the Istituto Nazionale di Alta Matematica (INdAM). S.B. acknowledges the support by the DFG through the International Research Training Group IGDK 1754 ``Optimization and Numerical Analysis for Partial Differential Equations with Nonsmooth Structures''.

\providecommand{\bysame}{\leavevmode\hbox to3em{\hrulefill}\thinspace}
\providecommand{\MR}{\relax\ifhmode\unskip\space\fi MR }
\providecommand{\MRhref}[2]{%
  \href{http://www.ams.org/mathscinet-getitem?mr=#1}{#2}
}
\providecommand{\href}[2]{#2}

\end{document}